\documentclass[a4paper,12pt]{amsart}

\usepackage{amssymb,amsmath,amsthm,latexsym} \usepackage{mathrsfs} 
\usepackage[dvipsnames]{xcolor}
\usepackage{marginnote}

\usepackage{graphicx} 
\usepackage{hyperref} 
\usepackage{mdwlist} \usepackage{enumerate} \usepackage{enumitem} \usepackage{mathtools,dsfont,wasysym} \usepackage{stmaryrd} 
\usepackage{hyperref} 
\usepackage{cleveref}
\usepackage{multicol} 
\hypersetup{colorlinks=true, linkcolor=blue, citecolor=red, urlcolor=cyan}

\newtheorem{theorem}{Theorem}[section]
\newtheorem{lemma}[theorem]{Lemma}

\newtheorem{proposition}[theorem]{Proposition}

\newtheorem{corollary}[theorem]{Corollary}

\newtheorem{theoremA}{Theorem} 
\newtheorem{corollaryA}{Corollary} 
\newtheorem{corollaryB}{Corollary} 
\newtheorem{theoremB}{Theorem}
\newtheorem{corollaryC}{Corollary} 
\newtheorem{corollaryD}{Corollary}

\theoremstyle{definition}
\newtheorem{definition}[theorem]{Definition}

\theoremstyle{remark}
\newtheorem{remark}[theorem]{Remark}
\numberwithin{equation}{section}

\newcommand{\R}{\mathbb{R}}
\newcommand{\C}{\mathbb{C}}
\newcommand{\N}{\mathbb{N}}
\newcommand{\K}{\mathbb{K}}

\newcommand{\Lin}{\mathcal{L}}

\DeclareMathOperator{\dist}{dist\,}

\DeclareMathOperator{\id}{Id}

\DeclareMathOperator{\Orb}{Orb}
\DeclareMathOperator{\Span}{span}

\newcommand{\nn}[1]{{\left\vert\kern-0.25ex\left\vert\kern-0.25ex\left\vert #1 
		\right\vert\kern-0.25ex\right\vert\kern-0.25ex\right\vert}}
\renewcommand{\geq}{\geqslant}
\renewcommand{\leq}{\leqslant}

\newcommand{\NA}{\operatorname{NA}}

\newcommand{\eps}{\varepsilon}

\newcommand{\vertiii}[1]{{\left\vert\kern-0.25ex\left\vert\kern-0.25ex\left\vert #1 
    \right\vert\kern-0.25ex\right\vert\kern-0.25ex\right\vert}}

\newcounter{smallromans}

	{\end{list}}

\makeatletter
\renewcommand{\tocsection}[3]{%
	\indentlabel{\@ifnotempty{#2}{\bfseries\ignorespaces#1 #2\quad}}\bfseries#3}
\renewcommand{\tocsubsection}[3]{%
	\indentlabel{\@ifnotempty{#2}{\ignorespaces#1 #2\quad}}#3}

\newcommand\@dotsep{4.5}
\def\@tocline#1#2#3#4#5#6#7{\relax
	\ifnum #1>\c@tocdepth 
	\else
	\par \addpenalty\@secpenalty\addvspace{#2}%
	\begingroup \hyphenpenalty\@M
	\@ifempty{#4}{%
		\@tempdima\csname r@tocindent\number#1\endcsname\relax
	}{%
		\@tempdima#4\relax
	}%
	\parindent\z@ \leftskip#3\relax \advance\leftskip\@tempdima\relax
	\rightskip\@pnumwidth plus1em \parfillskip-\@pnumwidth
	#5\leavevmode\hskip-\@tempdima{#6}\nobreak
	\leaders\hbox{$\m@th\mkern \@dotsep mu\hbox{.}\mkern \@dotsep mu$}\hfill
	\nobreak
	\hbox to\@pnumwidth{\@tocpagenum{\ifnum#1=1\bfseries\fi#7}}\par
	\nobreak
	\endgroup
	\fi}
\AtBeginDocument{%
	\expandafter\renewcommand\csname r@tocindent0\endcsname{0pt}
}
\def\l@subsection{\@tocline{2}{0pt}{2.5pc}{5pc}{}}
\makeatother

\usepackage{todonotes}

\begin{document}
	
	\title[Linear structures in the set of non-norm-attaining operators]{Linear structures in the set of non-norm-attaining operators on Banach spaces}
	
\author[S.~Dantas]{Sheldon Dantas}
\address[S.~Dantas]{Czech Technical University in Prague, FEE, Department of Mathematics, Technická 2, 16627, Prague 6, Czech Republic \newline
\href{https://orcid.org/0000-0001-8117-3760}{ORCID: \texttt{0000-0001-8117-3760}}}
\email{\texttt{sheldon.dantas@fel.cvut.cz}}
\urladdr{www.sheldondantas.com}

	\author[Falcó]{Javier Falcó}
	\address[Falcó]{Departamento de Análisis Matemático, Facultad de Ciencias Matemáticas, Universidad de Valencia, Doctor Moliner 50, 46100 Burjasot (Valencia), Spain
		\newline
		\href{http://orcid.org/0000-0000-0000-0000}{ORCID: \texttt{0000-0003-2240-2855} }}
	\email{\texttt{francisco.j.falco@uv.es}}
	
	\author[Jung]{Mingu Jung}
        \address[Jung]{Department of Mathematics \& Research Institute for Natural Sciences, Hanyang University, 04763 Seoul, Republic of Korea \newline
\href{https://orcid.org/0000-0003-2240-2855}{ORCID: \texttt{0000-0003-2240-2855}}}

	\email{\texttt{mingujung@hanyang.ac.kr}}

	\author[Rodríguez-Vidanes]{Daniel L. Rodríguez-Vidanes}
	\address[Rodríguez-Vidanes]{Grupo de Investigación de Análisis Matemático y Aplicaciones (AMA), Departamento de Matemática Aplicada a la Ingeniería Industrial, E.T.S.I.D.I., Ronda de Valencia 3, Universidad Politécnica de Madrid,
		Madrid, 28012, Spain \newline
		\href{https://orcid.org/0000-0002-1016-096X}{ORCID: \texttt{0000-0002-1016-096X} }}
	\email{\texttt{dl.rodriguez.vidanes@upm.es}}

	\begin{abstract} We study large linear structures inside sets arising in the theory of norm-attaining operators. We provide several results in the context of lineability, spaceability, maximal-spaceability, and $(\alpha, \beta)$-spaceability for sets of non-norm-attaining bounded linear operators whenever such sets are nonempty. To be more specific, we show that if $Y$ is a strictly convex renorming of $c_0 (\Gamma)$, then the set 
    \[
    \mathcal{L}(c_0 (\Gamma),Y)\setminus \overline{\NA (c_0 (\Gamma), Y)}
    \] 
    is $2^{|\Gamma|}$-spaceable. We also prove that 
    \[
    \mathcal{L}(d_* (w,1) ,\ell_p )\setminus \overline{\NA (d_* (w,1) ,\ell_p )}
    \]
    is maximal-spaceable. Finally, we establish that whenever the set of non-norm-attaining operators from a Banach space $X$ into $\ell_p (\Gamma)$ (respectively, $c_0 (\Gamma)$) is nonempty, it contains a subspace linearly isometric to $\ell_p(\Gamma)$ (respectively, $c_0 (\Gamma)$). These results extend and complement several known results in the literature concerning large linear structures in sets of non-norm-attaining operators. Our results are obtained in a more general framework involving group-invariant operators, which allows us to treat classical spaces of operators as special cases.	\end{abstract}

	\thanks{ }
	
	
	\maketitle
	
	\thispagestyle{plain}

	\section{Introduction}
The presence of large linear structures inside highly non-linear sets has become a central topic in modern analysis. This phenomenon, commonly referred to as \emph{lineability}, seeks to determine whether sets lacking any apparent algebraic structure may nevertheless contain infinite-dimensional linear subspaces. Since the term lineability was coined by Gurariy in the early 2000s \cite{GQ}, the study of lineability has been explored in a wide variety of contexts, revealing that many pathological families of functions and operators exhibit an unexpectedly rich linear geometry. Early examples of this idea can be found in a theorem of Levine and Milman \cite{LM}, which shows that the set of functions of bounded variation on $[0,1]$ does not contain closed infinite-dimensional subspaces of $C[0,1]$ endowed with the supremum norm.

Motivated by the highly non-linear nature of the subset of norm-attaining operators, it is natural to ask whether its complement may contain large closed infinite-dimensional linear structures. In particular, one may consider the classes of operators that do not attain their norms or that cannot be approximated by norm-attaining operators. The main purpose of this paper is to investigate the existence of such large closed subspaces in these sets.

Given Banach spaces $X$ and $Y$, we denote by $\mathcal L(X,Y)$ the Banach space of all bounded linear operators from $X$ to $Y$, and $\NA(X,Y)$ the set of all operators in $\mathcal{L}(X,Y)$ that attain their norms. In general, $\NA(X,Y)$ fails to be a linear subspace of $\mathcal L (X,Y)$. For instance, the set $\NA(\ell_1, \K)$ coincides with 
    \[
    \{x \in (\ell_1)^*=\ell_{\infty}: \|x\|_{\infty} = \max_{n \in \N} |x_n| \}
    \]
    which is not a subspace of $\ell_\infty$, although it contains $c_0$. More strikingly, Rmoutil proved that the set $\NA(X, \mathbb{R})$ may contain no $2$-dimensional subspaces \cite{Rm}, thereby resolving an open problem posed by Bandyopadhyay and Godefroy in \cite{BG}. Furthermore, Martín generalized Rmoutil's result in \cite{M2} by showing that, for any $n\in \mathbb N$, there is a Banach space $X_n$ such that $\NA(X_n,\mathbb R)$ contains vector subspaces of dimension $n$ but does not contain vector subspaces of dimension $n+1$. More recently, Martín showed in \cite{Martin2025} that $\NA(X,\mathbb R)$ may fail to contain any non-trivial cone.

   These phenomena are strongly influenced by the geometry of the underlying norm. In fact, García-Pacheco and Puglisi proved in \cite{GP} that every real infinite-dimensional Banach space $X$ admits an equivalent norm $\vertiii{\cdot}$ such that $\NA((X, \vertiii{\cdot}),\mathbb R)$ contains vector subspaces of infinite-dimension, illustrating that the linear structure of $\NA(X,\mathbb R)$ can be considerably unstable under renorming.
   Moreover, such structural properties also have geometric consequences. For instance, if $X$ is a smooth Banach space, then closed subspaces in $\NA(X, \R)$ must be rotund \cite[Theorem 3.4]{AAAG2007}.

      The linear structure of $\NA(X,\mathbb R)$ is also closely related to the dual geometry of Banach spaces. A classical theorem of Petun\={\i}n and Pl\={\i}\v{c}ko \cite{PP} shows that if $X$ is separable and $W\subseteq \NA(X,\mathbb R)$ is a closed separating subspace, then $W$ is an isometric predual of $X$. This connection highlights that lineability phenomena in sets of norm-attaining functionals may reflect deep structural properties of the space, a perspective further developed in a subsequent work of Bandyopadhyay and Godefroy \cite{BG}. They also provided renorming techniques ensuring that $\NA(X,\mathbb R)\cup \{0\}$ contains infinite-dimensional closed subspaces when $B_{X^*}$ is $w^*$-sequentially compact \cite[Theorem 2.12]{BG}.

    On the other hand, the set $X^* \setminus \NA(X,\mathbb{R})$ of all non-norm-attaining functionals on $X$ together with $\{0\}$ may contain large closed linear structures. In particular, under suitable assumptions on compact Hausdorff spaces $K$ and measures $\mu$, both $[{C}(K)^* \setminus \NA({C}(K), \R)]\cup\{0\}$ and $[L_1(\mu)^* \setminus \NA(L_1(\mu), \R)]\cup\{0\}$ contain infinite-dimensional \textit{closed} linear subspaces  \cite[Theorems 2.5 and 2.7]{AAAG2007}.

	In the vector-valued setting, Pellegrino and Teixeira \cite{PT} investigated lineability properties related to the set $\NA(X, \ell_p)$ and its complement $\mathcal{L}(X, \ell_p) \setminus \NA(X, \ell_p)$. They showed that, whenever $1 \leq p < \infty$ and $Y$ contains an isometric copy of $\ell_p$, the set 
    \[
    \NA^{x_0}(X, Y) := \{ T \in \NA(X,Y): \|T\|=\|Tx_0\| \}
    \]
    together with $\{0\}$ contains an infinite-dimensional subspace \cite[Proposition 6]{PT}. Under the same assumptions, the same holds for the set $[\mathcal{L}(X, Y) \setminus \NA(X,Y)] \cup \{0\}$ whenever it is non-empty \cite[Proposition 7]{PT} (we also refer the interested reader to \cite{DJM} for additional conditions ensuring the existence of non-norm-attaining operators). In the same spirit, it was shown in \cite[Theorem 2.3]{FGMR} that there are Banach spaces $X$ and $Y$ such that the subset of all compact operators which cannot be approximated by norm-attaining operators together with the null operator also contains infinite-dimensional closed subspaces.

In this paper we investigate lineability and spaceability phenomena for the class of non-norm-attaining operators. Our results provide new examples and extend several known results in the literature on norm-attainment theory, placing them within a broader framework extending \cite{AFM, DFJ, F, FGJM}. More precisely, we formulate our results in terms of group-invariant operators, which arise naturally in a setting that extends the classical theory of bounded linear operators.

In what follows, we summarize our main results and provide some comments on them. For the terminology $\mu$-spaceable, where $\mu$ is a cardinal, we refer the reader to Definition~\ref{lineability-properties} below. The symbol $d_{*}(w,1)$ stands for the predual of a Lorentz sequence space (see Subsection \ref{basic-notation}). 

    \smallskip
    
	\noindent
	{\bf Main Results}: Let $X$ and $Y$ be Banach spaces, and let $\Gamma$ be any infinite set. 
	\begin{enumerate}[label=\textbf{(M\arabic*)}]
 \itemsep0.3em
		\item \label{M1} If $Y$ is a strictly convex renorming of $c_0(\Gamma)$, then
		\begin{equation*}
			\mathcal{L}(c_0(\Gamma), Y) \setminus \overline{\NA(c_0(\Gamma), Y)}
		\end{equation*}
		is $2^{|\Gamma|}$-spaceable in $\mathcal{L}(c_0(\Gamma), Y)$. In particular, if $\Gamma = \mathbb N$, it is maximal-spaceable (specifically, it is $\mathfrak{c}$-spaceable) in $\mathcal{L}(c_0, Y)$. 
		
		\item \label{M2} Suppose that $w = (1/n)_{n=1}^{\infty} \in c_0$ and $1 < p < \infty$. Then,
		\begin{equation*}
			\mathcal{L}(d_{*}(w,1), \ell_p) \setminus \overline{\NA(d_{*}(w,1), \ell_p)}
		\end{equation*}
		is maximal-spaceable (specifically, it is $\mathfrak c$-spaceable) in $\mathcal{L}(d_{*}(w,1), \ell_p)$.

		\item \label{M3} 
        Let $\Gamma$ be an infinite set and let $1 \leq p < \infty$.
        \begin{enumerate}
            \item  If $\mathcal{L}(X, \ell_p (\Gamma)) \setminus \NA(X, \ell_p (\Gamma))$ is nonempty, then the subset $\mathcal{L}(X, \ell_p (\Gamma)) \setminus \NA(X, \ell_p (\Gamma))$ contains an isometric copy of $\ell_p (\Gamma)$. 
            \item If $\mathcal{L}(X, c_0(\Gamma)) \setminus \NA(X, c_0 (\Gamma))$ is nonempty, then the subset $\mathcal{L}(X, c_0(\Gamma)) \setminus \NA(X, c_0 (\Gamma))$ contains an isometric copy of $c_0(\Gamma)$.
            
        \end{enumerate}
	\end{enumerate}

The result \ref{M1} is obtained as a particular case of Theorem~\ref{thmA}, while \ref{M2} is a consequence of Theorem~\ref{thmB}. The statement in \ref{M3} follows from Theorems~\ref{thm:lqgamma} and \ref{thm:c_0-target}. We briefly discuss some relevant observations concerning the results \ref{M1}--\ref{M3}. First, the result \ref{M1} should be compared to the classical result due to J. Lindenstrauss in \cite{L}, where he gave a negative answer to the Bishop and Phelps question \cite{BP} on whether every bounded linear operator can be approximated by norm-attaining operators. Our result \ref{M1} yields an abundant supply of such operators, in the precise sense that they form an infinite-dimensional subspace with this property. Lindenstrauss’s work initiated extensive research on Bishop–Phelps type theorems for operators. 

A central notion that Lindenstrauss introduced was the so-called property B: a Banach space $Y$ satisfies {\it Lindenstrauss property B} if $\NA(Z,Y)$ is dense in $\mathcal{L}(Z, Y)$ for every Banach space $Z$. Among the notable negative results is the theorem of Gowers \cite{G}, which shows that the classical spaces $\ell_p$, for $1<p<\infty$, fail to satisfy property~B. The result \ref{M2} above complements this phenomenon by providing infinite-dimensional linear structures entirely formed by bounded linear operators taking values in $\ell_p$-spaces that cannot be approximated by norm-attaining operators. 

Finally, the result for $\ell_p(\Gamma)$ in item \ref{M3} extends a result of  Pellegrino and Teixeira \cite[Proposition 7]{PT}, while the corresponding result for $c_0(\Gamma)$ provides, to the best of our knowledge, the first example of this phenomenon when the range space is a $c_0(\Gamma)$-space.

Let us emphasize once again that the results \ref{M1}, \ref{M2} and \ref{M3} will be formulated in terms of $G$-invariant operators, where $G$ is a group acting continuously on the Banach space $X$ by linear isometries (see Subsections \ref{section:groups} and \ref{subsection:orbits}). In particular, by taking $G$ to be the trivial group $\{\id\}$, where $\id$ denotes the identity mapping on $X$, one recovers the statements presented in this introduction as a special case. We believe that this general framework both unifies and extends these phenomena, providing a broader perspective.

	\section{Preliminary Material}\label{sec:pre}

	In this second section, we introduce the notation and basic results needed so that the reader can follow the paper without having to consult multiple references. More specifically, we present the definitions related to lineability as well as the necessary concepts from set theory and the theory of invariant operators on Banach spaces under the action of a compact group $G$.

	\subsection{Basic Notation} \label{basic-notation} Throughout the paper, we use standard notation from Banach space theory (see, for instance, \cite{FHHMZ}). We denote by $X^*$, $B_X$, and $S_X$ the dual space, the unit ball, and the unit sphere of a Banach space $X$ over a field $\K$, either $\R$ or $\C$, respectively. The symbol $\mathcal{L}(X,Y)$ stands for the Banach space of all bounded linear operators from $X$ to $Y$. To simplify the notation, we write $\mathcal{L}(X)$ instead of $\mathcal{L}(X,X)$ when $X=Y$. We say that an operator $T \in \mathcal{L}(X,Y)$ {\it attains its norm}, or it is {\it norm-attaining} if there exists $x_0 \in S_X$ such that $\|T\| = \|T(x_0)\|$. We denote by $\NA(X, Y)$ the subset of $\mathcal{L}(X, Y)$ consisting of all bounded linear operators that attain their norms. The topological aspects of $\NA(X,Y)$ have been studied extensively in the literature; see, for example, the classical references \cite{B, JW, L, S, uhl}. As mentioned in the introduction and as simple examples show, $\NA(X,Y)$ is not, in general, a linear subspace of $\mathcal{L}(X,Y)$.
	
	As we will work with the predual of the Lorentz sequence spaces $d_{*}(w, 1)$ in Section \ref{main}, we briefly recall its definition. The predual  $d_*(w,1)$ is defined as 
	\[
	d_{*}(w,1)=\left\{x=(x_i)_{i=1}^{\infty} \in c_0: \lim_{k\to\infty}\frac{\sum_{i=1}^k[x_i]}{\sum_{i=1}^k w_i} =0 \right\}
	\] 
	where $([x_i])_{i=1}^{\infty}$ is the non-increasing rearrangement of $(|x|_i)_{i=1}^{\infty}$ endowed with the norm 
	\[
	\Vert x\Vert=\sup_k \frac{\sum_{i=1}^k[x_i]}{\sum_{i=1}^k w_i}.
\]
We refer the interested reader to \cite{Gar, Wer} for further details.

	\subsection{Lineability} Let us now provide the necessary background on lineability used throughout this paper. We introduce five related concepts in the following definition. We refer the reader to the monograph \cite{ABPS}. In what follows, $\dim(V)$ denotes the algebraic dimension of a topological vector space $V$ over $\mathbb K$.

	\begin{definition} \label{lineability-properties} Let $V$ be a topological vector space, $M$ a subset of $V$, and $\mu$ a cardinal number. We say that $M$ is
		\begin{itemize}
            \item[(i)] {\it lineable} if $M \cup \{0\}$ has an infinite-dimensional vector subspace of $V$.
		\item[(ii)] $\mu$-{\it lineable} if $M \cup \{0\}$ has a vector subspace of $V$ of dimension $\mu$.
            \item[(iii)] {\it spaceable} if $M \cup \{0\}$ has a closed infinite-dimensional vector subspace of $V$.
		\item[(iv)] $\mu$-{\it spaceable} if $M \cup \{0\}$ has a closed vector subspace of $V$ of dimension $\mu$. 
		\item[(v)] {\it maximal-spaceable} if $M$ is $\dim(V)$-spaceable. 
		\end{itemize}
	\end{definition}
	
    It was initially expected that negative results would be more common than positive ones when studying the lineability and spaceability of certain sets, since many of these non-linear sets are described as “monsters.” However, experience has shown that the opposite usually holds in the case of norm-attaining operators (and different fields), and results such as Rmoutil’s counterexample \cite{Rm} are in fact rare. As a related concept, V.V. Fávaro, D. Pellegrino, and D. Tomaz in \cite{FPT} introduced the more restricted notion of $(\alpha, \beta)$-spaceability. 
    	
	\begin{definition} \label{alpha-beta} Let $V$ be a topological vector space, $M$ a subset of $V$, and $\alpha\leq \beta$ two cardinal numbers. 
		We say that $M$ is $(\alpha,\beta)$-{\it spaceable} if $M$ is $\alpha$-lineable and for each $\alpha$-dimensional vector subspace $V_\alpha$ of $V$ with $V_\alpha \subseteq M\cup \{0\}$, there is a closed $\beta$-dimensional subspace $V_\beta$ of $V$ such that
		\begin{equation*} 
			V_\alpha \subseteq V_\beta \subseteq M\cup \{0\}.
		\end{equation*} 
	\end{definition}

Observe that $(\alpha,\alpha)$-spaceability implies $\alpha$-spaceability, but the converse is not true in general. For instance, it is known that the set $L_p[0,1] \setminus \bigcup_{q \in (1,\infty)} L_q [0,1]$ is $\mathfrak c$-spaceable for any $p>0$ \cite{BFPS} and, as an immediate consequence of \cite[Corollary~2.4]{FPRR}, the set $L_p[0,1] \setminus \bigcup_{q \in (1,\infty)} L_q [0,1]$ is not $(\mathfrak c,\mathfrak c)$-spaceable. It is worth mentioning that G. Araújo, A. Barbosa, A. Raposo Jr., and G. Ribeiro proved in \cite[Theorem~3]{ABRR} that $L_p[0,1] \setminus \bigcup_{q \in (1,\infty)} L_q [0,1]$ is $(\alpha,\mathfrak c)$-spaceable if and only if $\alpha < \aleph_0$ reducing this phenomenon to the countable case.

    From our theorems, we obtain results on the $(\alpha, \beta)$-spaceability of sets of non-norm-attaining operators. To this end, we use the following result as a tool in Section \ref{main}.
    If $A$ and $B$ are subsets of a vector space, then $A$ is said to be {\it stronger} than $B$ whenever $A+B \subseteq A$. Recall that an {\it $F$-space} is a topological vector space whose topology can be defined by a complete translation-invariant metric. 
	
	\begin{theorem} \cite[Theorem 2.1]{FPRR} \label{pellegrino} Let $\alpha \geq \aleph_0$ and $V$ be an $F$-space. Let $A, B$ be subsets of $V$ such that $A$ is $\alpha$-lineable and $B$ is $1$-lineable. If $A \cap B = \emptyset$ and $A$ is stronger than $B$, then $A$ is not $(\alpha, \beta)$-spaceable regardless of  the cardinal number $\beta$. 	
	\end{theorem} 
	
	For more results on $(\alpha,\beta)$-spaceability for families of operators, see \cite{DFPR,DR}.
	
    
	
	\subsection{Set Theory} We use standard notations and terminology from set theory (see, for instance, \cite{C}). 
    All the results in the present paper are within the standard framework of ZFC. The cardinality of a set $A$ will be denoted $|A|$. Ordinal numbers are identified with the sets of their predecessors, and cardinal numbers with initial ordinals. We denote by $\aleph_0$ and $\mathfrak c$ the cardinality of $\N$ and the continuum, respectively.
	
	The proofs of Theorems~\ref{thmA} and~\ref{thmB} rely on the \emph{Fichtenholz–Kantorovich–Hausdorff theorem} (see \cite{FK, Ha}). For more information on the various applications of the Fichtenholz-Kantorovich-Hausdorff theorem on the subject of lineability, we refer the reader to \cite{BBG, FMRS}.
	
	To state this result, we recall the notion of an independent family of subsets of a given nonempty set $\Gamma$. To simplify the notation, if $\Gamma$ is a nonempty set and $M\subseteq\Gamma$, then we write $M^1 = M$ and $M^0 = \Gamma \setminus M$. 
	
	\begin{definition} Let $\Gamma$ be a nonempty set. We say that a family $\mathcal{Y}$ of subsets of $\Gamma$ is {\it independent} if for any finite sequence of pairwise distinct sets $Y_1, \ldots, Y_n \in \mathcal{Y}$ and any $\eps_1, \ldots, \eps_n \in \{0,1\}$ we have that
		\begin{equation*}
			Y_1^{\eps_1} \cap \cdots \cap Y_n^{\eps_n} \not=\emptyset.
		\end{equation*}
	\end{definition}
	
	The Fichtenholz-Kantorovich-Hausdorff theorem reads as follows.
	
	\begin{theorem}[Fichtenholz-Kantorovich-Hausdorff theorem] \label{FKH} If $\Gamma$ is a set of infinite cardinality $\kappa$, then there is a family of independent subsets $\mathcal{Y}$ of $\Gamma$ of cardinality $2^{\kappa}$. 
	\end{theorem}
	
	
	\begin{remark} \label{FKH-sets-are-infinite} When one applies the Fichtenholz-Kantorovich-Hausdorff theorem, one gets a family $\mathcal{Y}$ of $2^{\kappa}$-many subsets of nonempty sets such that $Y_1^{\eps_1} \cap \cdots \cap Y_n^{\eps_n} \not= \emptyset$ whenever $Y_1, \ldots, Y_n \in \mathcal{Y}$ and $\eps_1,\ldots,\eps_n \in \{0,1\}$. As a matter of fact, more can be said. Indeed, the sets $Y_1^{\eps_1} \cap \cdots \cap Y_n^{\eps_n}$ besides being nonempty are infinite. For a simple proof of this fact, we refer the reader to the observation immediately after \cite[Definition 1.3]{FRST}.
	\end{remark}

	\subsection{$G$-symmetrization} \label{section:groups} Our main results are formulated in a broader framework in terms of group-invariant bounded linear operators. This approach subsumes the classical case of $\mathcal{L}(X,Y)$ and its subsets, while also providing a unified perspective that extends beyond it. In this subsection, we recall the background and notions related to group-invariant operators that will be needed later. 

	  
	Throughout the paper, $G \subseteq \mathcal{L}(X)$ will always be a topological subgroup of $\mathcal{L}(X)$ endowed with the relative topology inherited from the strong operator topology. Moreover we assume that the action of $G$ is compatible with the norm of $X$, that is, $\|g(x)\| = \|x\|$ for every $x \in X$ and $g \in G$. Equivalently, $G$ can be viewed as a subgroup of the group $\text{Iso}(X)$ of all linear isometries on $X$. For simplicity, we will often simply say that $G \subseteq \Lin(X)$ is a group.

\begin{definition}
Let $X$ and $Z$ be Banach spaces, and $G\subseteq \mathcal{L}(X)$ be a group. 
    \begin{itemize} 
 \setlength\itemsep{0.3em}
		\item[(i)] $T \in \mathcal{L}(X, Z)$ is \emph{$G$-invariant} if $T(g(x)) = T(x)$ for every $x \in X$ and $g \in G$. 
		\item[(ii)] $K \subseteq X$ is \emph{$G$-invariant} if $g(K) = K$ for every $g \in G$. 
		\item[(iii)] $x \in X$ is \emph{$G$-invariant} if the singleton $\{x\}$ is $G$-invariant. 
	\end{itemize} 
\end{definition}
We denote by $X_G$ the subspace of all $G$-invariant points $x$ in $X$. We also write $\mathcal{L}_G (X, Z)$ to denote the space of all $G$-invariant operators from $X$ into $Z$, namely, 
	\begin{equation*} 
		\mathcal{L}_G(X, Z) := \{T \in \mathcal{L}(X, Z): T \ \mbox{is} \ G \mbox{-invariant} \}.
	\end{equation*} 

   Let $Y$ be a subspace of $X$. If $Y$ is $G$-invariant, then $[G]:=\{[g]:g\in G\}$, where 
	\[
	[g] (x+Y) = g(x) +Y 
	\]
	is a well-defined group contained in $\Lin (X/Y)$. 
	 Notice that
	\begin{equation}\label{eq:isometric2}
		\Lin_{[G]} (X/Y, \mathbb{K}) = Y^\perp \cap \Lin_G (X,\mathbb{K})
	\end{equation}
 isometrically. Indeed, the mapping $\varphi \in \Lin_{[G]} (X/Y, \mathbb{K}) \mapsto \widetilde{\varphi} \in X^*$ defined by $\widetilde{\varphi} (x):=\varphi (x+Y)$, provides the desired isometric isomorphism in \eqref{eq:isometric2}. 
	Note that $[G]$ is a compact group whenever $G$ is compact.
	
    	If the group $G \subseteq \mathcal{L}(X)$ is assumed to be compact, then we can consider the {\it $G$-symmetrization} of a point $x$ in $X$, denoted by $\overline{x}$. Namely, $\overline{x}$ is the point in $X$ defined by 
	\begin{equation} \label{symmetrization} 
		\overline{x} := \int_G g(x) \, d\mu(g), 
	\end{equation} 
	where $\mu$ is the normalized Haar measure on the group $G$. Notice that $\|\overline{x}\|\leq \|x\|$, and if $x$ is already $G$-invariant, then $\overline{x} = x$. This, in particular, implies that the mapping $P:x \mapsto \overline{x}$ is a norm-one linear projection on $X$ with range $X_G$, that is, $X_G$ is a 1-complemented subspace of $X$. We refer to $P$ as the canonical norm-one $G$-symmetrization projection.
    

 
  Let us record some useful facts in the following remark which will be used in Section \ref{NA-functionals}.
  
	\begin{remark}\label{rem:identifications} Let $G$ be a compact group. 
    \begin{itemize}
		\itemsep0.3em
		\item[(i)] Let $Z$ be a Banach space. Then, the restriction mapping 
		\begin{equation}\label{eq:isometric1}
			\Lin_G (X, Z) \rightarrow \Lin (X_G, Z), \quad T \mapsto T \vert_{X_G}
		\end{equation}
		can be verified to be an isometric isomorphism. This shows that the sets $\NA_G(X,Z)$ and $\NA (X_G, Z)$ can be identified. 
        
		
		\item[(ii)]  Consider $Y \subseteq X$ (thus, $Y_G \subseteq X_G)$. Then 
		\[
		Y_G^\perp \rightarrow Y^\perp \cap \Lin_G (X,\mathbb{K}), \quad y^* \mapsto y^* \circ P
		\]
		is an isometric isomorphism. From this, we shall consider $Y_G^\perp$ as a subspace of $\Lin _G (X, \mathbb{K})$ without mentioning it explicitly. 
		Having  \eqref{eq:isometric2} and \eqref{eq:isometric1} in mind, we have
		\begin{equation}\label{eq:isometric3}
			(X_G / Y_G)^* = Y_G^\perp = \Lin_{[G]} (X/Y, \mathbb{K}) = \Lin ( (X/Y)_{[G]}, \mathbb{K}).
		\end{equation}
  Moreover, if we denote by $q_G$ the canonical quotient map from $X_G$ to $X_G/Y_G$, then for any $\phi \in (X_G/Y_G)^*$, there exists $y^* \in Y_G^\perp$ such that $y^* = \phi \circ q_G$. 
		\item[(iii)]  In fact, for $Y \subseteq X$, we have that the following is an isometric isomorphism:
		\begin{equation*}\label{eq:isometric_last}
		(X/Y)_{[G]} \rightarrow X_G/Y_G, \quad x+Y \mapsto P(x) + Y_G.
		\end{equation*}
		From this equation, \eqref{eq:isometric3} is obtained automatically. 
	\end{itemize}
    \end{remark}


    \subsection{The set of finite $G$-orbits} \label{subsection:orbits}
	In addition to $G$-symmetrization (see \eqref{symmetrization} above), the notion of orbits plays a key role in our discussion (see, for instance, the proof of Theorem \ref{thmA}). The {\it $G$-orbit} (or, simply \emph{orbit} when the context is clear) of a point $x \in X$ is the set $\Orb(x) = \{g(x): g \in G\}$. In the context of Banach spaces indexed by an infinite set $\Gamma$, a natural class of groups to consider is permutation groups acting on $\Gamma$. 
	Such an action partitions the set $\Gamma$ into pairwise disjoint sets. We denote by $\Orb(G, \Gamma)$ the family of orbits induced by the group action $G$ on $\Gamma$. 

	We use the following notation. Given $N \in \mathbb{N}$, we set
	\begin{equation*}
		\mathcal{F}_{G, \Gamma}^N := \{ H \in \Orb(G, \Gamma): |H| \leq N \} \ \ \mbox{and} \ \ 
		\mathcal{F}_{G, \Gamma} := \{ H \in \Orb(G, \Gamma): |H| < \aleph_0 \}. 
	\end{equation*}
	Observe that 
	\begin{equation*} 
		\mathcal{F}_{G, \Gamma} = \bigcup_{N\in \mathbb N} \mathcal{F}_{G, \Gamma}^N.
	\end{equation*} 
	For later use (see, for instance, the proofs of Theorems \ref{thmA} and \ref{thmB}), we introduce the following notation. Given a permutation group $G$ acting on a set $\Gamma$, we define 
	\begin{equation*}  \label{G-infty}
		|G|_\infty := \begin{cases} 
			\sup\left\{ \left|\mathcal{F}_{G, \Gamma}^N \right| : N\in \N \ \ \mbox{and}  \ \ |\mathcal{F}_{G, \Gamma}^N|\geq \aleph_0 \right\}, & \text{if } |\mathcal{F}_{G, \Gamma}^N|\geq \aleph_0 \text{ for some } N\in \N, \\
			0, & \text{otherwise}.
		\end{cases}
	\end{equation*}
    The quantity $|G|_\infty$ reflects how many finite orbits the action of $G$ can produce. The following elementary properties are immediate.
	\begin{itemize}
		\setlength\itemsep{0.3em}
		\item[(i)] If $G$ is the trivial group (i.e., $G=\{\text{Id}\}$), then $|G|_\infty = |\Gamma|$. 
		\item[(ii)] If $|\mathcal{F}_{G, \Gamma}^N|\geq \aleph_0$ for some $N\in \N$, then $|G|_\infty = |\mathcal F_{G,\Gamma}|$. 
		\item[(iii)] If $\mathcal{F}_{G,\mathbb{N}}$ is infinite and there exists $N \in \mathbb{N}$ such that $|H|\leq N$ for every $H \in \mathcal{F}_{G,\mathbb{N}}$ (as assumed in \cite[Propositions 4.1, 4.2 and 4.4]{DFJ}), then we have that $|G|_\infty = |\mathcal{F}_{G, \mathbb{N}}^N| >0$.
	\end{itemize} 
    Observe that for any infinite set $\Gamma$, the trivial group acting on $\Gamma$ satisfies (iii).
    However, there are infinite sets $\Gamma$ and groups of permutation acting on $\Gamma$ that do not satisfy the conclusion of property (iii), as shown in Lemma \ref{extremecase} below.

    \begin{lemma}\label{extremecase}
        There is a set $\Gamma$ and a group of permutations acting on $\Gamma$ such that $|G|_\infty > |\mathcal F_{G,\Gamma}^N|$ for any $N \in \mathbb N$.
    \end{lemma}

    \begin{proof}
        Let $\omega_\alpha$ denote the $\alpha$-th infinite initial ordinal.
		Take $\Gamma = \omega_\omega$ and let us construct a group $G$ of permutations acting on the elements of $\Gamma$ satisfying 
  \begin{enumerate}[label=($*$)]
\item \text{for every $N \in \mathbb{N}$, if $H \in \mathcal F_{G,\omega_\omega}^N$, then $H \subset \omega_{N-1}$.} \label{eq:group}
  \end{enumerate}
            
        Since 
            $$
            |\omega_0|=\aleph_0=|\omega_0 \times \{0\}|,
            $$
        there is a bijection $f_0\colon \omega_0 \times \{0\} \to \omega_0$.
        Hence, observe that 
            $$
            \Omega_0 := \{ f_0[\{n\} \times \{0\}] \colon n < \omega_0 \}
            $$ 
        forms a partition of $\omega_0$ such that each set of the partition has cardinality $1$.
        Analogously, for every $N\in \mathbb N$, we have 
            $$
            |\omega_N \setminus \omega_{N-1}| = \aleph_N = |\omega_N \times \{0,1, \ldots,N\}|.
            $$
        Thus, there is a bijection $f_N \colon \omega_N \times \{0,1,\ldots,N\} \to \omega_N \setminus \omega_{N-1}$.
        In this last case, we obtain that 
            $$
            \Omega_N := \{ f_N[\{\alpha\} \times \{0,1,\ldots,N\}] \colon \alpha < \omega_N \}
            $$
        is a partition of $\omega_N \setminus \omega_{N-1}$ such that each set has cardinality $N+1$.
        Now set
            $$
            \Omega := \bigcup_{N=0}^\infty \Omega_N.
            $$
        Note that in particular $\Omega$ is a partition of $\omega_\omega$ with $\Omega_N \cap \Omega_M = \varnothing$ for any $N,M\in \mathbb N \cup \{0\}$ distinct.
        Moreover, for any $W \in \Omega$, there are by construction unique $N \in \mathbb N \cup \{0\}$ and $\alpha < \omega_N$ such that 
            $$
            W = f_N[\{\alpha\} \times \{0,1,\ldots,N\}].
            $$
        So, by definition of $f_N$, there are unique $\omega_{N-1} \leq \beta_0 < \beta_1 < \cdots < \beta_N < \omega_N$ if $N>0$ and $\beta_0 < \omega_0$ otherwise such that 
            $$
            W = \{ \beta_0,\beta_1,\ldots,\beta_N \}.
            $$
        Now construct $g_W \in G$ acting only on $W$ as follows: if $N=0$, then $g_W (\beta_0)=\beta_0$; if $N>0$, then $g_W (\beta_i) = \beta_{i+1}$ for any $0\leq i<N$ and $g_W (\beta_N) = \beta_0$.
        Take 
            $$
            G:=\{e\} \cup \{g_W \colon W \in \Omega \},
            $$
        where $g_W$ is extended to $\Omega$ by fixing the elements outside $W$, and $e$ satisfies that $e(\beta)=\beta$ for any $\beta < \omega_\omega$ (i.e., the neutral element).
        Then, we have constructed a group $G$ acting on $\omega_\omega$ such that
         each $H \in \Orb(G,\omega_\omega)$ with $|H|=N$ for some $N\in \mathbb N$ satisfies that $H \subset \omega_{N-1}$, 
        finishing the construction.

        Let $G$ be a group of permutations on $\omega_\omega$ that satisfies \ref{eq:group}.
		Then, notice that 
            $$
            \left| \mathcal F_{G,\omega_\omega}^N \right| = \aleph_{N-1} < \aleph_\omega
            $$ 
        for every $N\in \mathbb N$, and $|G|_\infty = \aleph_\omega$.
        The proof is finished.
    \end{proof}



	\section{The Set of Norm-Attaining Functionals} \label{NA-functionals}	

As we have mentioned before, the study of linear structures inside $\NA(X,\mathbb{K})$ for a Banach space $X$ is closely related to the theory of isometric duality. Indeed, if $X$ is a dual Banach space, then $\NA(X,\mathbb{K})$ contains an isometric copy of a predual of $X$. This line of investigation is also closely related to the theory of proximinality as we will see in a moment. Recall that a closed subspace $Y$ of $X$ is said to be a {\it proximinal} subspace of $X$ if for every $x \in X$, there exists $y \in Y$ such that $\|x-y\| = d(x, Y)$. It is well-known that $x^* \in \NA(X,\mathbb{K})$ if and only if the hyperplane $\ker x^*$ is proximinal in $X$. More generally, it is known that if $Y$ is a finite-codimensional subspace of $X$ and proximinal in $X$, then $Y^\perp\subseteq \NA(X,\mathbb{K})$ (and the converse also holds in some cases; see \cite{BG, NR, Rm} and references therein).

We begin with some basic observations on proximinality in the $G$-invariant setting. The proofs are omitted since they are immediate consequences of the identifications $\mathcal{L}_G(X,\mathbb{K})=\mathcal{L} (X_G,\mathbb{K})$ and $\NA_G (X,\mathbb{K})=\NA(X_G,\mathbb{K})$ (see Remark \ref{rem:identifications}), together with known results (see \cite[Lemma 2.2]{BG} and \cite[Lemma 3.3]{Rm}). 


	\begin{proposition}
		Let $X$ be a Banach space, $Y \subseteq X$ a closed subspace and $G\subseteq \Lin(X)$ a compact group of isometries. Suppose that $Y$ is $G$-invariant. 
		\begin{enumerate}[label=(\roman*)]
		    \itemsep0.3em
			\item If $X_G/Y_G$ is strictly convex and $Y_G^\perp \subseteq \NA_G (X, \mathbb{K})$, then $Y_G$ is proximinal in $X_G$.
			\item If $X_G/Y_G$ is isometrically isomorphic to a dual Banach space $Z^*$ and $Y_G$ is proximinal in $X_G$, then $Z \subseteq \NA_G (X, \mathbb{K})$.
			\item Suppose that $Y_G$ is proximinal in $X_G$. Then $X_G/Y_G$ is reflexive if and only if $Y_G^\perp \subseteq \NA_G (X, \mathbb{K}).$
		\end{enumerate}
	\end{proposition}
	
    
 
 
	
	Suppose that $X$ is a Banach space such that $B_{X^*}$ is $w^*$-sequentially compact. Note that this property is inherited to a $1$-complemented subspace of $X$. Having this in mind, the following version of \cite[Lemma 2.7]{BG} for group invariant linear functionals is clear. 
	
	\begin{proposition} Let $X$ be a Banach space such that $B_{X^*}$ is w*-sequentially compact, and let $G$ be a compact group of isometries in $\Lin(X)$. Let $M \subseteq \NA_G (X, \mathbb{K})$ be a norm closed separable subspace. Then, we have that $M^*$ is the canonical quotient of $X_G$.
	\end{proposition}

    The lineability and spaceability of norm-attaining functionals may exhibit different behaviors in the framework of group-invariant linear functionals. We recall the following results from \cite{AAAG2007, GP2010}. Let $K$ be an infinite compact Hausdorff space, and let $\mu$ be a $\sigma$-finite measure such that $L_1 (\mu)$ is infinite-dimensional. Then,
	\begin{itemize}
 \itemsep0.3em
		\item[(i)] $\NA(C(K), \mathbb{R})$ is lineable (\cite[Theorem 2.1]{AAAG2007}) and not spaceable when $K$ is scattered (see \cite[Remark~2.2]{AAAG2007} and \cite[Proposition~2.20]{BG}).
		\item[(ii)] $C(K)^* \setminus \NA (C(K), \mathbb{R})$ is lineable. If $K$ has a non-trivial convergent sequence, then $C(K)^* \setminus \NA (C(K), \mathbb{R})$ is spaceable (see \cite[Theorem 2.5]{AAAG2007}). 
		\item[(iii)] $\NA (L_1 (\mu), \mathbb{R})$ is spaceable (see \cite[Theorem 2.4]{GP2010}).
		\item[(iv)] $L_1(\mu)^* \setminus \NA(L_1 (\mu), \mathbb{R})$ is spaceable (see \cite[Theorem 2.7]{AAAG2007}). 
	\end{itemize}

    However, it is known from \cite[Propositions 4.5 and 4.7]{DFJ} that there exist groups $G \subseteq \Lin (C[0,1])$ and $H\subseteq \Lin (L_1 [0,1])$ such that for every Banach space $Y$ with a Schauder basis, the following holds true  
	\begin{equation*} 
		\NA_G (C[0,1], Y) = \Lin_G (C[0,1], Y)  \text{ and } \NA_H (L_1[0,1], Y) = \Lin_H (L_1[0,1], Y). 
	\end{equation*}

In a related direction, \cite[Theorem 3.1]{AAAG2007} shows that if a Banach space $X$ has a monotone Schauder basis, then $\NA(X,\mathbb{K})$ is lineable. We establish Theorem \ref{prop:schauder} below, which provides a group-invariant analogue of this result.

The natural assumption we make here is that the Schauder basis is $1$-symmetric, ensuring that the natural group action arising from permutations is a group of isometries. For completeness, recall that a Schauder basis $(e_n)_{n=1}^\infty$ of a Banach space $X$ is said to be \textit{symmetric} if, for any permutation $\pi$ of the natural numbers, $(e_{\pi (n)})_{n=1}^\infty$ is equivalent to $(e_n)_{n=1}^\infty$. Moreover, $K = \sup_{\theta, \pi} \|M_\theta\, V_\pi\| < \infty$, where $M_\theta$ and $V_\pi$ are defined by 
	\begin{equation*} 
		M_\theta \left(\sum_{n \in \N} a_n e_n \right) = \sum_{n \in \N} a_n \theta_n e_n \quad \text{and} \quad V_\pi \left(\sum_{n \in \N} a_n e_n \right)  = \sum_{n \in \N} a_n e_{\pi(n)}
	\end{equation*} 
	for every choice of signs $\theta=(\theta_n)_{n\in \N}$ and permutation $\pi$ of the natural numbers. The number $K$ is called the \textit{symmetric constant} of $(e_n)_{n=1}^\infty$. A symmetric Schauder basis with symmetric constant $K$ is referred to as a \textit{$K$-symmetric basis}. 
    
 The basic notation required for the following result is introduced in Subsection \ref{subsection:orbits} above.


	
 

    
	\begin{theorem} \label{prop:schauder} Let $X$ be an infinite-dimensional Banach space with $1$-symmetric basis $(e_n)_{n=1}^\infty$ and $G$ a group of permutations of the natural numbers. 
		\begin{itemize}
			\itemsep0.3em
			\item[\textup{(i)}] If $\mathcal{F}_{G,\mathbb{N}}$ is finite, then $\NA_G (X, \mathbb{K})$ contains a vector space of dimension at least $|\mathcal{F}_{G,\mathbb{N}}|$. If, in addition, the basis $(e_n)_{n=1}^\infty$ is weakly null, then 
			\[
			\NA_G (X, \mathbb{K}) = \mathcal{L}_G (X, \mathbb{K}) 
			\]
			is a finite-dimensional space.
			\item[\textup{(ii)}] If $|G|_\infty >0$, then $\NA_G (X, \mathbb{K})$ is $\aleph_0$-lineable.        
		\end{itemize}
	\end{theorem}
	
	\begin{proof}
		Let us denote by $(e_n^*)_{n=1}^\infty$ the corresponding coefficient functionals. 
        
 We start by proving (i). For simplicity, let $\mathcal{F}_{G,\mathbb{N}} = \{H_1,\ldots,H_N\}$, and consider
		\[
		M :=  \left\{ \sum_{i \in H_1} a_1 e_i^* + \cdots + \sum_{i \in H_N} a_N e_i^*  : a_i \in \mathbb{K}, \, i=1,\ldots,N \right\}.
		\]
		Note that every element in $M$ is a $G$-invariant functional.
		Moreover, as $(e_n)_{n=1}^\infty$ is $1$-symmetric (hence its basis constant is $1$), we have 
   \begin{equation*}     
        B_{\text{span}\{e_i : i \in H_1 \cup \cdots \cup H_N \}} = Q(B_X), 
    \end{equation*}
        where $Q$ is the canonical projection from $X$ onto the finite-dimensional subspace $\text{span}\{e_i : i \in H_1 \, \cup \, \cdots \, \cup \, H_N \}$. This implies that every element in $M$ is norm-attaining; hence $M \subseteq \NA_G (X,\mathbb{K})$. Next, assume that $(e_n)_{n=1}^\infty$ converges weakly to $0$. Let $x^* \in \mathcal{L}_G(X,\mathbb{K})$ and $n_0 \not\in H_1 \cup \cdots \cup H_N$ be fixed. As the orbit $\{g(n_0):g\in G\}$ contains an arbitrarily large natural number, and $x^* (e_{n_0}) = x^* (e_{g(n_0)})$ for every $g \in G$, we deduce that $x^* (e_{n_0})=0$. It follows that the functional $x^*$ depends only on the coordinates in $H_1, \ldots, H_N$. Also, as $\{x^*(e_i) : i \in H_j\}$ is a singleton for each $j=1,\ldots, N$, the element $x^*$ belongs to $M$ and this proves (i). 
		
	In order to prove (ii), let $N \in \mathbb{N}$ be such that $\mathcal{F}_{G,\mathbb{N}}^N$ is infinite and let $\{ H_n : n \in \mathbb{N}\}$ be an enumeration of $\mathcal{F}_{G,\mathbb{N}}^N$. 
		Consider
		\[
		M = \left\{ \sum_{i \in H_1} a_1 e_i^* + \cdots + \sum_{i \in H_n} a_n e_i^*  : a_i \in \mathbb{K},  \, n \in \mathbb{N}\right\} \subseteq X^*.
		\]
From the argument in (i), every element in $M$ is $G$-invariant and also norm-attaining. 	\end{proof}
	

	
	\section{The Set of Non-Norm-Attaining Operators} \label{main}
	
The purpose of this section is to investigate the linear structure of the set of bounded linear operators that cannot be approximated by norm-attaining operators. In other words, whenever this set is nonempty, we study the algebraic structure of 
\begin{equation*} 
\mathcal{L}_G (X,Y) \setminus \overline{\NA_G (X,Y)}. 
\end{equation*}

As mentioned earlier, we consider the pairs of Banach spaces $(E, F) = (c_0, Y)$ and $(d_* (w,1), \ell_p)$ for which the closure of the set of non-norm-attaining operators from $E$ into $F$ is nonempty (we refer the reader to Subsection \ref{basic-notation} for the definition of $d_*(w,1)$ and to Subsection \ref{subsection:orbits} for the notation $\mathcal{F}_{G, \N}$ which will appear in the upcoming results). As a consequence of the results obtained in this section, we derive results \ref{M1}, \ref{M2}, and \ref{M3} stated in the Main Results.

Before entering into the details, we recall some results from \cite{DFJ} that illustrate how group actions may influence the density of norm-attaining operators.

	\begin{theorem} \cite[Proposition 4.1]{DFJ} \label{previous-paper-results1}  Let $G$ be a group of permutations of the natural numbers and let $Y$ be a strictly convex renorming of $c_0$. 
		\begin{itemize}
			\item[\textup{(i)}] If $\mathcal{F}_{G, \N}$ is finite, then 
			\begin{equation} \label{non-NA-0}
				\NA_G(c_0, Y) = \mathcal{L}_G(c_0, Y). 
			\end{equation}
			\item[\textup{(ii)}]   If $\mathcal{F}_{G, \N}$ is infinite and there is $N\in \mathbb N$ such that $|H|\leq N$ for every $H \in \mathcal{F}_{G, \N}$, then 
			\begin{equation} \label{non-NA-1}
				\overline{\NA_G(c_0, Y)}  \not= \mathcal{L}_G(c_0, Y).
			\end{equation}
		\end{itemize} 
	\end{theorem}

	\begin{theorem}[{\cite[Proposition 4.2]{DFJ}}] \label{thm:previous-paper-results2} 
    Let $G$ be a group of permutations of the natural numbers and $Y$ be a Banach space. 
		Suppose that $\mathcal{F}_{G, \N}$ is infinite and there is $N\in \mathbb N$ such that $|H|\leq N$ for every $H \in \mathcal{F}_{G, \N}$. If $w=(1/n)_{n=1}^{\infty} \in c_0$ and $1<p<\infty$, then 
			\begin{equation} \label{non-NA-2}
				\overline{\NA_G(d_{*}(w,1), \ell_p)}  \not= \mathcal{L}_G(d_{*}(w,1), \ell_p).
			\end{equation}
	\end{theorem} 
	
	In Theorem \ref{previous-paper-results1}, if $G$ is taken to be the trivial group $\{\text{Id}\}$, then $\mathcal{F}_{G,\mathbb{N}}$ is infinite and $|H| = 1$ for every $H \in \mathcal{F}_{G,\N}$. In this case, the theorem reduces to the classical result of Lindenstrauss \cite[Proposition 4]{L} on the existence of operators which cannot be approximated by norm-attaining operators.
 For the same reason, the result \eqref{non-NA-2} can be viewed as a group invariant extension of Gowers’ result \cite[Theorem, p.~149]{G}.

	\begin{remark}
The first, second, and third authors did not state that the orbits $H\in \mathcal F_{G,\N}$ need to satisfy a uniform boundedness condition in \cite[Proposition 4.1.(2)]{DFJ}. 
		However, the operator $T$ in that proof requires the cardinality of orbits in $\mathcal{F}_{G, \mathbb{N}}$ to be uniformly bounded so that $T$ is bounded.
		Here we fix this typo by stating the correct version of the result in Theorem~\ref{thm:previous-paper-results2}.
	\end{remark}

\subsection{From $c_0$ to its strictly convex renorming}	

We present our first main result for the case where the domain space is $c_0(\Gamma)$ along with their consequences (for comparison, see (\ref{non-NA-0}) and (\ref{non-NA-1}) in Theorem \ref{previous-paper-results1} above). For the notation $\mathcal{F}_{G, \Gamma}$ below, we refer the reader once again to Subsection \ref{subsection:orbits}.

	\begin{theoremA}\label{thmA}
		Let $\Gamma$ be an infinite set, $G$ a group of permutations of the elements of $\Gamma$ acting on $c_0(\Gamma)$, and $Y$ be a Banach space.
		\begin{enumerate}
		    [label={\textup{(\arabic*)}}]
			\setlength\itemsep{0.3em}
			\item \label{thmA-1} If $\mathcal{F}_{G, \Gamma}$ is finite, then every $G$-invariant operator is norm-attaining, that is, 
			\begin{equation*}     
				\mathcal{L}_G(c_0(\Gamma), Y) = \NA_G(c_0(\Gamma), Y).	
			\end{equation*} 
			
			\item \label{thmA-2} If $|G|_\infty>0$ and $Y$ is a strictly convex renorming of $c_0(\Gamma)$, then the set
			\begin{equation*} 
				\mathcal{L}_G(c_0(\Gamma), Y) \setminus \overline{\NA_G(c_0(\Gamma), Y)}
			\end{equation*} 
			is $2^{|G|_\infty}$-spaceable in $\mathcal{L}_G(c_0(\Gamma), Y)$. 
		\end{enumerate}
	\end{theoremA}

        Theorem~\ref{thmA} extends and strengthens \cite[Proposition~4.1]{DFJ} in the setting of  $c_0(\Gamma)$ with an infinite index set $\Gamma$. It is worth noting that every $c_0(\Gamma)$, for any set $\Gamma$, can be renormed to become strictly convex. In fact, there is an explicit equivalent norm on $c_0(\Gamma)$, known as Day’s norm, which is locally uniformly rotund (hence strictly convex) \cite[Theorem~7.3]{DGZ}.

    	Before presenting the proof of Theorem \ref{thmA} and its corollaries, we prove the following lemma, which is of independent interest. It generalizes \cite[Lemma~2]{M}, and the proof follows the same idea as in the original argument. We include it here for the sake of completeness.

    For each $\gamma \in \Gamma$, let us denote by $e_\gamma \in c_0 (\Gamma)$ the function on $\Gamma$ which has value $1$ at $\gamma$ and $0$ elsewhere. 
	
	\begin{lemma}\label{martin}
		Let $\Gamma$ be an infinite set.
		If $X$ is a closed subspace of $c_0(\Gamma)$ and $Y$ is a strictly convex Banach space, then every non-zero norm-attaining operator from $X$ to $Y$ depends only on finitely many elements $e_\gamma$'s.
	\end{lemma}
	
	\begin{proof}
		Let $T\in \NA(X,Y)$ and $x_0\in  B_X$ such that $\|T(x_0)\|=\|T\|=1$.
		Since $x_0\in c_0(\Gamma)$, the set $\Gamma_{1/2}:=\{\gamma\in \Gamma : |x_0(\gamma)|\geq 1/2 \}$ is finite.
		Consider
		$$
		Z:=\left\{x\in X : x(\gamma)=0,\ \forall \,\gamma\in \Gamma_{1/2} \right\}.
		$$
		Observe that every $z\in Z$ with $\|z\|\leq 1/2$ satisfies that $\|x_0\pm z\|\leq 1$.
		Thus, 
		\begin{equation}\label{equ:2}
			\|T(x_0)\pm T(z)\| = \|T(x_0\pm z)\| \leq \|T\| \|x_0\pm z\| \leq 1.
		\end{equation}
Being $Y$ strictly convex, $\|T(x_0)\|=1$ and \eqref{equ:2} yield $T(z)=0$.
		It follows that $T$ vanishes on $Z$. Consequently, $T(e_\gamma) = 0$ whenever $\gamma \not\in \Gamma_{1/2}$. 
	\end{proof}

Let us also recall the following well-known inequality for complex numbers, which will be used in the proofs of Theorems \ref{thmA} and \ref{thmB}.

\begin{lemma}[\mbox{\cite[Lemma 6.3]{R}}]  \label{complex-case} If $z_1, \ldots, z_n \in \C$, then there exists a subset $S \subseteq \{1, \ldots, n\}$ for which 
		\begin{equation*}
			\frac{1}{\pi} \sum_{k=1}^{n} |z_k| \leq \Big|  \sum_{j \in S} z_k \Big| \leq \sum_{k=1}^n |z_k|.
		\end{equation*}
	\end{lemma}
	
	Finally, we are ready to prove Theorem \ref{thmA}. For item (II), we apply Fichtenholz-Kantorovich-Hausdorff theorem which is stated in Theorem \ref{FKH} above.

\begin{proof}[Proof of Theorem \ref{thmA}] 
We prove first Theorem \hyperref[thmA-1]{A.(1)}, whose argument is inspired by the proof of \cite[Proposition~4.1.(1)]{DFJ}. Let $Y$ be a Banach space and $T \in \mathcal{L}_G(c_0(\Gamma), Y)$ be fixed. We show that $T$ is norm-attaining. Fix $\gamma \in \Gamma$ and assume that $\Orb(\gamma)$ is infinite. If $\gamma_1, \ldots, \gamma_m$ are distinct elements of $\Orb(\gamma)$ with $m\in\mathbb{N}$, then the element $x_m := \sum_{j=1}^m e_{\gamma_j}$ has norm one in $c_0(\Gamma)$. Since $T$ is $G$-invariant, we have that $\{ T(e_{\gamma'}): \gamma'  \in \Orb(\gamma) \} = \{ T(e_\gamma) \}$. This implies that 
\begin{equation*}
T(x_m) = \sum_{j=1}^m T(e_{\gamma_j}) = \sum_{j=1}^m T(e_{\gamma}) = m \cdot T(e_{\gamma}).
\end{equation*}
Since $m$ is arbitrary, $T(e_{\gamma}) = 0$ whenever $\Orb(\gamma)$ is infinite.

Now, write $\mathcal{F}_{G, \Gamma} = \{H_1, \ldots, H_N\}$, where each $H_k \subseteq \Gamma$ is finite. Consider $F:= \bigcup_{k=1}^N H_k$ and $P_F: c_0(\Gamma) \rightarrow c_0(\Gamma)$ to be the coordinate projection onto $F$. By the first part of the proof, we have that $T(e_{\gamma}) = 0$ for all $\gamma \not\in F$. It follows that 
\[
T(x) = T(P_F (x)) + T ( ( \id -P_F ) (x)) = T(P_F (x)) 
\]
for every $x \in c_0 (\Gamma)$. Thus, $T$ only depends on finitely many coordinates. Consequently, $T \in \NA_G(c_0(\Gamma), Y)$.

Let us now prove Theorem \hyperref[thmA-2]{A.(2)}. We distinguish two cases.

\noindent		
\textbf{Case 1.}  $|G|_\infty \neq |\mathcal F_{G,\Gamma}^N|$ for every $N \in \mathbb{N}$.

\noindent
\textbf{Case 2.}  There exists $N\in \N$ such that $|G|_\infty = |\mathcal F_{G,\Gamma}^N|$.

In both cases we will construct a family ${\mathcal F}$ of subsets of $\Gamma$ such that $|{\mathcal F}| = |G|_\infty$. 
		
For Case 1, by the definition of $|G|_\infty$ (see Subsection \ref{subsection:orbits}), there exists a strictly increasing sequence of natural numbers $(N_k)_{k=1}^\infty$ such that 
\begin{equation*} 
\aleph_0 \leq \left|\mathcal F_{G,\Gamma}^{N_k}\right| < \left|\mathcal F_{G,\Gamma}^{N_{k+1}}\right| \ \ \ \mbox{and} \ \ \ |G|_\infty = \sup \left\{ \left|\mathcal F_{G,\Gamma}^{N_k}\right| : k\in \N \right\}. 
\end{equation*} 
Observe that, for each $k\in \N$, 
\begin{equation*} 
\left|\mathcal F_{G,\Gamma}^{N_{k+1}} \setminus \mathcal F_{G,\Gamma}^{N_k}\right| = \left|\mathcal F_{G,\Gamma}^{N_{k+1}}\right|.
\end{equation*} 
Moreover, for all $i \neq j$ in $\N$, 
\begin{equation*} 
\left(\mathcal F_{G,\Gamma}^{N_{i+1}} \setminus \mathcal F_{G,\Gamma}^{N_i}\right) \cap \left( \mathcal F_{G,\Gamma}^{N_{j+1}} \setminus \mathcal F_{G,\Gamma}^{N_j}\right) = \emptyset.
\end{equation*} 
In this case, we set 
\begin{equation*} 
\mathcal{F} = \bigcup_{k\in \N} \left(\mathcal F_{G,\Gamma}^{N_{k+1}} \setminus \mathcal F_{G,\Gamma}^{N_k}\right).
\end{equation*}

		For Case 2, for such a natural number $N$, let us define $\mathcal{F} := \mathcal{F}_{G,\Gamma}^N$.  
		
        Thus, in either case, we obtain $|{\mathcal F}| = |G|_\infty$.

		To shorten our notation, we will denote by $\Omega$ an enumeration, without repetition, of $\mathcal{F}$, that is $\mathcal{F}=\{H_\alpha:\alpha\in \Omega\}$.
		Clearly $|H_\alpha|<\aleph_0$ for any $\alpha\in \Omega$ and $H_\alpha \cap H_\beta = \emptyset$ for every $\alpha, \beta \in\Omega$ with $\alpha \neq \beta$. 	Observe that $|\Omega| = |G|_\infty \leq |\Gamma|$, so we can fix a one-to-one mapping $\varphi$ from $\Omega$ to $\Gamma$.

Write $Y=(c_0(\Gamma), \| \cdot \|_Y)$ for a strictly convex renorming of $c_0 (\Gamma)$.
We define a linear operator $T: c_0(\Gamma) \rightarrow Y$ as follows. For each element $x=\sum_{\gamma \in \Gamma} x_\gamma e_\gamma \in c_0(\Gamma)$, we set
\begin{equation*} 
		T(x)=\sum_{\gamma\in \Gamma} x_\gamma T(e_\gamma),
\end{equation*} 
where 
\begin{equation*} 
		T(e_\gamma) :=\begin{cases}
			\dfrac{1}{|H_\alpha|} \,e_{\varphi(\alpha)}, &	\text{if} \ \gamma\in H_\alpha \text{ for some } \alpha \in \Omega, \\
            0, &\text {otherwise}.
		\end{cases}
\end{equation*} 
		Equivalently, for every $x=\sum_{\gamma \in \Gamma} x_\gamma e_\gamma \in c_0(\Gamma)$, 
\begin{equation*} 
		T(x) = \sum_{\alpha \in \Omega} \left( \frac{1}{|H_\alpha|} \sum_{\gamma \in H_\alpha} x_\gamma \right) e_{\varphi(\alpha)}.
\end{equation*}

		\noindent
		{\bf Claim 1}: $T$ is well-defined (i.e., $T(x) \in c_0(\Gamma)$ for each $x \in c_0 (\Gamma)$).

Fix $x = \sum_{\gamma \in \Gamma} x_\gamma e_\gamma \in c_0(\Gamma)$ and $\varepsilon > 0$.
It suffices to prove that the set $$S:= \{\gamma \in \Gamma : |T(x)(\gamma)|\geq \varepsilon \}$$ is finite. To this end, define $\Gamma_{x,\eps} := \left\{\gamma \in \Gamma : |x_\gamma|\geq \varepsilon \right\}$. By the construction of $T$, we may assume that $x_\gamma = 0$ for $\gamma \not\in \bigcup_{\alpha \in \Omega} H_\alpha$. Since $\Gamma_{x,\eps}$ is finite, it follows that there exist $\alpha_1, \ldots, \alpha_k \in \Omega$ (with $k \in \mathbb{N}$) such that $\Gamma_{x,\eps} \subseteq \bigcup_{i=1}^k H_{\alpha_i}$. 

Let $\xi \in S$ be arbitrary. If $\xi \neq \varphi(\alpha)$ for every $\alpha \in \Omega$, then $T(x)(\xi)=0$ which contradicts $\xi \in S$. Hence there exists $\beta \in \Omega$ such that $\xi = \varphi(\beta)$. Assume that $|x_\gamma| < \varepsilon$ for every $\gamma \in H_\beta$, then 
		\begin{equation*}
			|T(x)(\xi)| = \Big| \frac{1}{|H_{\beta}|} \sum_{\gamma \in H_{\beta}} x_\gamma \Big| \leq \frac{1}{|H_{\beta}|} \sum_{\gamma \in H_{\beta}} |x_\gamma| < \frac{1}{|H_{\beta}|} \sum_{\gamma \in H_{\beta}} \varepsilon = \frac{1}{|H_{\beta}|} |H_{\beta}| \varepsilon = \varepsilon,
		\end{equation*}
        which contradicts $\xi \in S$. Therefore, there exists $\gamma \in H_\beta$ such that $|x_\gamma| \geq \varepsilon$. In particular, there exists $i \in \{1,\ldots, k\}$ such that $\gamma \in H_{\alpha_i}$ and hence $\beta = \alpha_i$. Consequently, 
        $$\xi =\varphi(\beta) \in \{\varphi(\alpha_1),\ldots, \varphi(\alpha_k)\}.
        $$
Thus, $S \subseteq \{\varphi(\alpha_1),\ldots, \varphi(\alpha_k)\}$ which completes the proof of the claim. 
        

		\noindent
		{\bf Claim 2}: $T$ is a $G$-invariant bounded linear operator from $c_0 (\Gamma)$ to $Y$.

Let $M > 0$ be such that 
\begin{equation} \label{eq:equivalent_norm_Y}
M^{-1} \|\cdot\|_{c_0(\Gamma)} \leq \|\cdot\|_Y \leq M \|\cdot\|_{c_0(\Gamma)}.
\end{equation} 
Fix $x = \sum_{\gamma \in \Gamma} x_\gamma e_\gamma \in c_0(\Gamma)$.	Since $|H_\alpha|<\aleph_0$ for every $\alpha \in \Omega$, there exists $\gamma_\alpha \in H_\alpha$ such that  $|x_{\gamma_\alpha}| = \underset{\gamma \in H_\alpha}{\sup } |x_\gamma|$ for each $\alpha \in \Omega$. Hence,
		\begin{align*}
			&\|T(x)\|_Y  \\ 
   &\leq M \left\|  \sum_{\alpha \in \Omega} \left( \frac{1}{|H_\alpha|} \sum_{\gamma \in H_\alpha} x_\gamma \right) e_{\varphi(\alpha)} \right\|_{c_0(\Gamma)} =  M \sup_{\alpha \in \Omega} \left| \frac{1}{|H_\alpha|} \sum_{\gamma \in H_\alpha} x_\gamma \right| \leq  M \sup_{\alpha \in \Omega} |x_{\gamma_\alpha}|,
		\end{align*}
		which implies that 
		$\|T(x)\|_Y \leq  M \|x\|_{c_0(\Gamma)}$. On the other hand,  notice that $T$ is linear and $G$-invariant by construction.

          \noindent	{\bf Claim 3}: There exists a family $\mathcal{Y}$ of subsets of $\Omega$ such that $\{T_F : F \in \mathcal{Y}\}$ is a linearly independent, uniformly bounded family of $G$-invariant bounded linear operators.

We apply Fichtenholz-Kantorovich-Hausdorff theorem (see Theorem \ref{FKH}) to the set $\Omega$ to obtain a family $\mathcal{Y}$ of independent subsets of $\Omega$ such that $|\mathcal{Y}|= 2^{|\Omega|} = 2^{|G|_\infty}$. We claim that this family $\mathcal{Y}$ meets our purpose. For each member of this family $F \in \mathcal{Y}$, we define the operator $T_F: c_0(\Gamma) \rightarrow Y$ by 
\begin{equation*} 
		T_F(x)=\sum_{\gamma\in \Gamma} x_\gamma T_F (e_\gamma),
\end{equation*} 
		where 
\begin{equation*} 
		T_F(e_\gamma) :=\begin{cases}
			\dfrac{1}{|H_\alpha|} \, e_{\varphi(\alpha)}, &	\text{if} \ \gamma\in H_\alpha \text{ for some } \alpha \in F, \\		0, &\text {otherwise}.
		\end{cases}
\end{equation*} 
As before, $T_F$ can be written as
\begin{equation*} 
		T_F(x) = \sum_{\alpha \in F} \left( \frac{1}{|H_\alpha|} \sum_{\gamma \in H_\alpha} x_\gamma \right) e_{\varphi(\alpha)}
\end{equation*} 
for any $x=\sum_{\gamma \in \Gamma} x_\gamma e_\gamma \in c_0(\Gamma)$.	It is immediate that $T_F$ is well-defined and linear. Moreover, $T_F$ is $G$-invariant and satisfies $\|T_F\| \leq \|T\|$ for every $F \in \mathcal{Y}$. Thus, the family $\{T_F: F \in \mathcal{Y} \}$ is a uniformly bounded family of $G$-invariant bounded linear operators. 
	
  It remains to prove that the family $\{T_F: F \in \mathcal{Y} \}$ is linearly independent. For this, suppose that $\sum_{i=1}^m a_i T_{F_i} = 0$ for some $m \in \N$, $a_1, \ldots, a_m \in \K$ and distinct sets $F_1, \ldots, F_m \in \mathcal{Y}$. As the sets $F_1,\ldots, F_m$ are independent subsets, we can guarantee the existence of $\alpha \in \Omega$ such that 
		\begin{equation*}
			\alpha \in F_1 \setminus \left( \bigcup_{i=2}^m F_i \right).
		\end{equation*}
		For any $\gamma \in H_{\alpha}$, we then have 
		\begin{equation*}
			0 = \left( \sum_{i=1}^m a_i T_{F_i} \right) (e_{\gamma}) = a_1 T_{F_1} (e_\gamma) = \frac{a_1}{|H_\alpha|} e_{\varphi (\alpha)}.
		\end{equation*}
		This, in particular, implies that $a_1 = 0$. Analogously, we can show that $a_i = 0$ for every $i=2,3,\ldots, m$; therefore the family $\{T_F: F \in \mathcal{Y} \}$ is linearly independent.

        \noindent
		{\bf Claim 4}: $\overline{\Span} \{T_F : F \in \mathcal{Y} \}   \subseteq \left(\mathcal{L}_G(c_0(\Gamma), Y) \setminus \overline{\NA_G(c_0(\Gamma), Y)} \right) \cup \{0\}.$




		Let now $S \in \overline{\Span} \{T_F : F \in \mathcal{Y} \} $ with $S \not= 0$. Without loss of generality, assume that $\|S\| = 1$. Therefore, there exist $a_1, \ldots, a_m \in \K$ with $a_j \not= 0$ for $j=1,\ldots, m$ and distinct sets $F_1, \ldots, F_m \in \mathcal{Y}$ such that 
		\begin{equation*}
			L := \sum_{j=1}^m a_j T_{F_j} \ \ \ \ \mbox{and} \ \ \ \|S - L\| < \frac{1}{2 M \pi  \|T\|}
		\end{equation*}
        		where $M>0$ is the constant in \eqref{eq:equivalent_norm_Y}.

		Without loss of generality, we may further assume that $\|L\| = 1$. 
		Now observe that 
		\begin{equation*}
			1 = \|L\| \leq \sum_{j=1}^m |a_j| \|T_{F_j}\| \leq \sum_{j=1}^m |a_j| \|T\| = \|T\| \sum_{j=1}^m |a_j|.
		\end{equation*}
		On the other hand, by Lemma \ref{complex-case}, there exists a subset $\Lambda \subseteq \{1, \ldots, m\}$ such that 
		\begin{equation*} \label{ineq2} 
			\Big| \sum_{s \in \Lambda} a_s \Big| \geq \frac{1}{\pi} \sum_{j=1}^m |a_j| \geq \frac{1}{\pi \|T\|}.
		\end{equation*}
		
		Consider a countably infinite set $\{\alpha_k : k \in \N \} \subseteq \bigcap_{s \in \Lambda} F_s \setminus \Big( \bigcup_{j \not\in \Lambda} F_j \Big)$ (see Remark \ref{FKH-sets-are-infinite}) and the element 		
        \begin{equation*}
			\mathsf{1}_k := \sum_{\gamma\in H_{\alpha_k}} e_\gamma 
		\end{equation*}
        for each $k \in \mathbb{N}$, 		which is clearly an element of $c_0 (\Gamma)$ since $|H_\alpha|$ is finite for every $\alpha \in \Omega$. Observe that 
		\begin{equation*}
			L(\mathsf{1}_k) =  \sum_{j\in \Lambda} a_j T_{F_j} (\mathsf{1}_k) = \sum_{j\in \Lambda} a_j \Big( \sum_{\gamma \in H_{\alpha_k}} T_{F_j} (e_\gamma)\Big) = \Big( \sum_{j\in \Lambda} a_j \Big) e_{\varphi(\alpha_k)}.
		\end{equation*}
Consequently, 	
\begin{equation}\label{eq:S_estimate}
			\|S(\mathsf{1}_{k})\| \geq \Big|  \sum_{j\in \Lambda} a_j  \Big| \|e_{\varphi(\alpha_k)}\|_Y - \|S-L\| > \frac{1}{M \pi \|T\|} -  \frac{1}{2 M \pi  \|T\|}  = \frac{1}{M \pi \|T\|}
		\end{equation} 
        for every $k \in \mathbb{N}$.

                On the other hand, if $R \in \NA_G (c_0(\Gamma), Y)$, then by Lemma \ref{martin}, there is a subset $\Gamma_0 \subseteq \Gamma$ so that $|\Gamma_0|$ is finite and $R(e_\gamma)=0$ whenever $\gamma \not\in \Gamma_0$. Since $H_\alpha \cap H_\beta = \emptyset$ whenever $\alpha \neq \beta$, we can find $k_0 \in \mathbb{N}$ such that $H_{\alpha_{k_0}} \cap \Gamma_0 = \emptyset$. Thus, $R(\mathsf{1}_{k_0}) = 0$, while $\| S(\mathsf{1}_{k_0}) \| \geq 1/M\pi\|T\|$ by \eqref{eq:S_estimate}. This shows that 
        \begin{equation*}
        \dist (L, \NA_G (c_0 (\Gamma),Y)) \geq \frac{1}{M \pi \|T\|}.    
        \end{equation*}
		This completes the proof of Claim~4.

		\noindent
		{\bf Claim 5}: $\dim (\overline{\Span} \{T_F : F \in \mathcal{Y} \}) = 2^{|G|_\infty}$.

		Note that $\Span \{T_F : F \in \mathcal{Y} \}$ is a vector space over $\mathbb K$ of dimension $2^{|G|_\infty} \geq \mathfrak c = |\mathbb K|$. Thus, 
		$$
		|\Span \{T_F : F \in \mathcal{Y} \}| = \max \{ \dim (\Span \{T_F : F \in \mathcal{Y} \}),|\mathbb K| \} = 2^{|G|_\infty}.
		$$
Moreover, for each $R \in \overline{\Span} \{T_F : F \in \mathcal{Y} \}$, we can assign a sequence $(R_n)_{n=1}^\infty \in (\Span \{T_F : F \in \mathcal{Y} \})^{\mathbb N}$ such that $\|R_n - R\| \rightarrow 0$ as $n \rightarrow \infty$. This correspondence yields an injective map from $\overline{\Span} \{T_F : F \in \mathcal{Y} \}$ into $(\Span \{T_F : F \in \mathcal{Y} \})^{\mathbb N}$; hence 
\[
|\overline{\Span} \{T_F : F \in \mathcal{Y} \}| \leq |\Span \{T_F : F \in \mathcal{Y} \})^{\mathbb N}| = \left( 2^{|G|_\infty} \right)^{\aleph_0} = 2^{|G|_\infty}. 
\]
Consequently,
\begin{align*}
2^{|G|_\infty} \leq \dim (\overline{\Span} \{T_F : F \in \mathcal{Y} \}) 
			\leq |\overline{\Span} \{T_F : F \in \mathcal{Y} \}| \leq 2^{|G|_\infty}
		\end{align*}
as we wanted to prove.
	\end{proof}

As far as we are aware, the above spaceability result is new even in the classical case $G=\{\mathrm{Id}\}$.
The following corollary corresponds to this particular case and is precisely \ref{M1}.


	\begin{corollaryA} \label{corbigA1} Let $\Gamma$ be an infinite set and $Y$ be a strictly convex renorming of $c_0(\Gamma)$. Then, the set
		\begin{equation*} 
			\mathcal{L}(c_0(\Gamma), Y) \setminus \overline{\NA(c_0(\Gamma), Y)}
		\end{equation*} 
		is $2^{|\Gamma|}$-spaceable in $\mathcal{L}(c_0(\Gamma), Y)$. In particular, in the case when $\Gamma=\mathbb{N}$, the set 
		\begin{equation*} 
			\mathcal{L}(c_0, Y) \setminus \overline{\NA(c_0, Y)}
		\end{equation*} 
		is maximal-spaceable (specifically $\mathfrak c$-spaceable) in $\mathcal{L}(c_0, Y)$.
	\end{corollaryA}


    


The next corollary shows that, although Theorem \ref{thmA} provides large linear structures inside $\mathcal{L}_G(c_0(\Gamma),Y)\setminus\overline{\NA_G(c_0(\Gamma),Y)}$, 
such structures cannot be arbitrarily large when measured in terms of $(\alpha,\beta)$-spaceability (see Definition \ref{alpha-beta}).

	\begin{corollaryB} \label{corbigA2} Let $\Gamma$ be an infinite set, $G$ be a group of permutations of the elements of $\Gamma$ acting on $c_0(\Gamma)$ such that $|G|_\infty>0$, $Y$ be a strictly convex renorming of $c_0(\Gamma)$, and let $\alpha$ be a cardinal number with $\alpha \geq \aleph_0$.
		\begin{enumerate}[label={\textup{(\arabic*)}}]
  \itemsep0.3em
			\item If $\alpha \leq 2^{|G|_\infty}$, then the set 
			\begin{equation*} 
				\mathcal{L}_G(c_0(\Gamma), Y) \setminus \overline{\NA_G(c_0(\Gamma), Y)}
			\end{equation*} 
			is not $(\alpha,\beta)$-spaceable regardless of the cardinal number $\beta \geq \alpha$. \label{corbigA2-1} 
			
			\item If $\alpha \leq 2^{|\Gamma|}$, then the set
			\begin{equation*} 
				\mathcal{L}(c_0(\Gamma), Y) \setminus \overline{\NA(c_0(\Gamma), Y)}
			\end{equation*} 
			is not $(\alpha,\beta)$-spaceable regardless of the cardinal number $\beta \geq \alpha$. \label{corbigA2-2}
			  
			\item The set
			\begin{equation*} 
				\mathcal{L}(c_0, Y) \setminus \overline{\NA(c_0, Y)}
			\end{equation*} 
			is not $(\alpha,\beta)$-spaceable regardless of the cardinal number $\beta \geq \alpha$. \label{corbigA2-3}
		\end{enumerate}   
	\end{corollaryB}

	\begin{proof} 
    We only prove the item \hyperref[corbigA2-1]{A2.(1)} as the items \hyperref[corbigA2-2]{A2.(2)} and \hyperref[corbigA2-3]{A2.(3)} are its consequences. Let 
\begin{equation*} 		
A :=\mathcal{L}_G(c_0(\Gamma), Y) \setminus \overline{\NA_G(c_0(\Gamma), Y)} \ \ \ \mbox{and} \ \ \  B := \overline{\NA_G(c_0(\Gamma), Y)}.
\end{equation*} 
As $\alpha \leq 2^{|G|_\infty}$, it is clear that $A$ is $\alpha$-lineable by Theorem \hyperref[thmA-2]{A.(2)} and $B$ is $1$-lineable with $A\cap B = \emptyset$. 

If $R$ and $S$ are elements of $\NA_G (c_0(\Gamma),Y)$, then there exist finite subsets $\Gamma_R$ and $\Gamma_{S}$ of $\Gamma$ such that $R$ depends only on $\{e_\gamma : \gamma \in \Gamma_R\}$ and $S$ depends only on $\{e_\gamma : \gamma \in \Gamma_{S}\}$. It follows that $R+S$ depends only on the elements in $\{e_\gamma : \gamma \in \Gamma_R \cup \Gamma_{S}\}$; hence $R+S \in \NA_G (c_0 (\Gamma), Y)$. Consequently, $B+B\subseteq B$. This proves that $A$ is stronger than $B$. 	\end{proof}

   
 \subsection{From Gowers space to $\ell_p$}
    Next, we present a further main result concerning the case in which the domain space is $d_{*}(w,1)$ and the target space is an $\ell_p$-space. 
 
 \begin{theoremB}\label{thmB}
		Let $G$ be a group of permutations of $\mathbb N$  with $|G|_\infty>0$  acting on $d_{*}(w,1)$.
		Let $w=(1/n)_{n=1}^{\infty} \in c_0$ and $1<p<\infty$. Then, the subset
			\begin{equation*}
				\mathcal{L}_G(d_{*}(w,1), \ell_p) \setminus \overline{\NA_G(d_{*}(w,1),\ell_p)}
			\end{equation*} 
			is maximal-spaceable (specifically, $\mathfrak c$-spaceable) in $\mathcal{L}_G(d_{*}(w,1), \ell_p)$. \label{thmB-1}
	\end{theoremB}

    \begin{proof}
	Since $|G|_\infty > 0$, we can find $N\in \N$ such that $\mathcal{F}_{G, \N}^N$ is infinite.
		Observe that $|\mathcal{F}_{G, \N}^N|=\aleph_0$.
		Let $\{H_n : n\in \N\}$ be an enumeration of $\mathcal{F}_{G, \N}^N$.
		For each $n \in \N$, we consider the finite-dimensional subspace of $d_{*}(w, 1)$ given by $X_n := \Span \{e_k: k \in H_n \}$.
		We define on $X_n$ the bounded linear functional $f_n:X_n \rightarrow \K$ given by 
		\begin{equation*}
			f_n(x) = \sum_{j \in H_n} \frac{e_j^*}{|H_n|} \ \ \ (x \in X_n)
		\end{equation*}
		where $(e_k^*)_{k=1}^\infty$ are the coordinate functionals of $(e_k)_{k=1}^\infty \subseteq d_{*}(w,1)$. Then, each $f_n$ is $G$-invariant and satisfies the inequality given by 
		\begin{equation*} \label{d-star-ineq-1}
			f_n(e_j) = \frac{1}{|H_n|} \geq \frac{1}{N} \quad \text{for every } j \in H_n \text{ and } n \in \N.
		\end{equation*}
Let us define the linear operator $T: d_{*}(w,1) \rightarrow \ell_p$ given by
		\[
		T (e_j) :=\begin{cases}
			(0,\ldots,0,\underbrace{f_n (e_j)}_{n ^{\text{th}}{\text{-term}}},0, \ldots), &	\text{if} \ j\in H_n \text{ for some } n \in \N, \\
			0, &\text {otherwise}.
		\end{cases}
		\]

		\noindent
		{\bf Claim 1}: $T$ is a bounded operator and $T \not\in \overline{\NA_G(d_{*}(w,1), \ell_p)} $.

		
		Indeed, for every $x=(x_1, x_2, \ldots) \in d_{*}(w,1)$, we have that 
		\begin{equation*}
			T(x) = \sum_{j=1}^{\infty} x_j T(e_j) = \sum_{n=1}^{\infty} \Big( 0,\ldots,0, \underbrace{\sum_{j \in H_n} x_j f_n (e_j)}_{n ^{\text{th}}{\text{-term}}},0, \ldots \Big).
		\end{equation*}
		By the Jensen's inequality, we have that 
		\begin{align*}
			\|T(x)\|_p^p &= \sum_{n=1}^{\infty} \Big| \sum_{j \in H_n} x_j f_n(e_j) \Big|^p  = \sum_{n=1}^{\infty} \left| \frac{\sum_{j \in H_n} x_j}{|H_n|} \right|^p \leq \sum_{n=1}^{\infty} \frac{1}{|H_n|}  \sum_{j \in H_n} |x_j|^p. 
		\end{align*}
		Now, since $|H_n| \geq 1$ for every $n \in \N$, 
		\begin{equation*}
			\|T(x)\|_p^p \leq \sum_{n=1}^{\infty} \sum_{j \in H_n} |x_j|^p \leq \sum_{j=1}^{\infty} |x_j|^p.
		\end{equation*}
		This implies that $T$ is bounded. Next, observe that if $S \in \NA (d_* (w,1), \ell_p)$, then there is $N \in \N$ such that $S(e_k) = 0$ for every $k > N$ (see the proof of \cite[Theorem, p.149]{G}).
		As $\| T (e_j) \| \geq N^{-1}$ for every $j \in H_n$ and $n \in \mathbb{N}$, we conclude that $T$ cannot be approximated by norm-attaining operators.

		Applying Fichtenholz-Kantorovich-Hausdorff theorem (Theorem \ref{FKH}) to $\mathbb{N}$, there exists a family $\mathcal{N}$ of independent subsets of $\N$ of cardinality $\mathfrak{c}$. For each $A \in \mathcal{N}$, we define the operator $T_A: d_{*}(w,1) \rightarrow \ell_p$ by 
		\[
		T_A(e_j) :=\begin{cases}
			(0,\ldots,0,\underbrace{f_n (e_j)}_{n ^{\text{th}}{\text{-term}}},0, \ldots), &	\text{if} \ j\in H_n \text{ for some } n \in A, \\
			0, &\text {otherwise}.
		\end{cases}
		\]
		Notice that each $T_A$ is well-defined and linear. Moreover, $T_A$ is $G$-invariant and satisfies that $\|T_A\| \leq \|T\|$ for every $A \in \mathcal{N}$, that is, the family $\{T_A: A \in \mathcal{N}\}$ is uniformly bounded. Moreover, the same argument as in the proof of Claim 3 in Theorem \ref{thmA} yields that the family $\{T_A: A \in \mathcal{N} \}$ is linearly independent.

		\noindent
		{\bf Claim 2}: We have that 
		\begin{equation*}
			\overline{\text{span}} \{T_A: A\in\mathcal{N}\} \subseteq \left( \mathcal{L}_G(d_{*}(w,1), \ell_p) \setminus \overline{\NA_G(d_{*}(w,1), \ell_p)} \right) \cup \{0\}.
		\end{equation*}

		Let $S \in \overline{\text{span}}\{T_A: A\in\mathcal{N}\} $ with $\|S\|=1$. Therefore, there are $a_1, \ldots, a_m \in \K$ with $a_i\not= 0$ for every $i=1,\ldots, m$ and $A_1, \ldots, A_m \in \mathcal{N}$ with $A_i \neq A_j$ if $i \neq j$ such that 
		\begin{equation} \label{d-star-ineq4}
			L := \sum_{i=1}^m a_i T_{A_i} \ \ \ \mbox{and} \ \ \ \|S - L\| < \frac{1}{2 \pi N \|T\|}.
		\end{equation}
		We also may assume that $\|L\| = 1$. 
		Now observe that 
		\begin{equation} \label{d-star-ineq-2} 
			1 = \|L\| \leq \sum_{i=1}^m |a_i| \|T_{A_i}\| \leq \sum_{i=1}^m |a_i| \|T\| = \|T\| \sum_{i=1}^m |a_i|.
		\end{equation}
		By Lemma \ref{complex-case} and \eqref{d-star-ineq-2}, there exists $S \subseteq \{1, \ldots, m\}$ such that 
		\begin{equation} \label{d-star-ineq-3}
			\Big| \sum_{i \in S} a_i \Big| \geq \frac{1}{\pi} \sum_{i=1}^m |a_i| \geq \frac{1}{\pi \|T\|}.
		\end{equation}
		Again by Remark \ref{FKH-sets-are-infinite}, we can write 
		\begin{equation*}
			\bigcap_{i \in S} A_i \setminus \Big( \bigcup_{j \not\in S} A_j \Big) = \{ l_k: k \in \N \}.
		\end{equation*}
		For every $k \in \N$, we fix $r_k \in H_{l_k}$. 
		Let us observe that $r_k \not= r_l$ for every $k \not= l$ since we have that $H_{n_k} \cap H_{n_l} = \emptyset$ for every $k \not= l$. 
		This shows that $\{r_k: k \in \N\}$ is infinite. 
		Thus, for every $k \in \N$, we have
		\begin{equation*}
			L(e_{r_k}) = \sum_{i \in S} a_i f_{l_k} (e_{r_k}) e_{l_k} = \Big( \sum_{i \in S} a_i \Big) f_{l_k} (e_{r_k}) e_{l_k} \in \ell_p.
		\end{equation*}
		Therefore, by (\ref{d-star-ineq4}) and (\ref{d-star-ineq-3}), we get that 
		\begin{equation*}
			\|S(e_{r_k})\|_p \geq \Big| \sum_{i \in S} a_i \Big|  |f_{l_k}(e_{r_k})| - \|S - L\| > \frac{1}{\pi N \|T\|} - \frac{1}{2 \pi N \|T\|}=\frac{1}{2 \pi N \|T\|} > 0.
		\end{equation*}
		This shows that $S \not\in \overline{\NA_G(d_{*}(w,1), \ell_p)} $ which proves the claim.
		
		
		\noindent
		{\bf Claim 3}: $\mathcal{L}_G(d_{*}(w,1), \ell_p) \setminus \overline{\NA_G(d_{*}(w,1),\ell_p)} $ is maximal-spaceable ($\mathfrak{c}$-spaceable).

        The claim follows from 
        $$
		\mathfrak c \leq \dim (\overline{\text{span}} \{T_A: A\in\mathcal{N}\}) \leq \dim (\mathcal{L}_G(d_{*}(w,1), \ell_p)) \leq \mathfrak c
		$$
since $d_* (w,1)$ and $\ell_p$ both are separable. This finishes the proof.
\end{proof}

        When $G= \{ \text{Id}\}$, Theorem \ref{thmB} reduces to the following result, which corresponds to \ref{M2}.
	
	\begin{corollaryC} \label{corbigB2} Let $w=(1/n)_{n=1}^{\infty} \in c_0$ and $1<p<\infty$. 
			Then, the set
			\begin{equation*} 
				\mathcal{L}(d_{*}(w,1), \ell_p) \setminus \overline{\NA(d_{*}(w,1),\ell_p)}
			\end{equation*} 
			is maximal-spaceable in $\mathcal{L} (d_* (w,1), \ell_p)$.
	\end{corollaryC}

    In the same spirit of Corollary \ref{corbigA2}, we have the following negative result regarding $(\alpha,\beta)$-spaceability.
    
	\begin{corollaryD} \label{corbigB2} Let $G$ be a group of permutations of $\mathbb N$ with $|G|_\infty>0$ and let $\alpha$ be a cardinal number such that $\alpha \geq \aleph_0$. Let $w=(1/n)_{n=1}^{\infty} \in c_0$ and $1<p<\infty$, and suppose that $G$ acts on $d_*(w,1)$. 
			Then, the set
			\begin{equation*} 
				\mathcal{L}_G(d_{*}(w,1), \ell_p) \setminus \overline{\NA_G(d_{*}(w,1),\ell_p)}
			\end{equation*} 
			is not $(\alpha,\beta)$-spaceable regardless of the cardinal number $\beta \geq \alpha$. \label{corbigB2-1}
	\end{corollaryD}

		\begin{proof}
        Fix a cardinal number $\beta \geq \alpha$. If $\alpha > \mathfrak c$, then $\mathcal{L}_G(d_{*}(w,1), \ell_p) \setminus \overline{\NA_G(d_{*}(w,1),\ell_p)}$ is not $(\alpha,\beta)$-spaceable since $|\mathcal{L}_G(d_{*}(w,1), \ell_p)| = \mathfrak c < \alpha$. Thus, assume that $\aleph_0 \leq \alpha \leq \mathfrak c$. In order to apply Theorem \ref{pellegrino}, let us set
		\begin{equation*} 
			A := \mathcal{L}_G(d_{*}(w,1), \ell_p) \setminus \overline{\NA_G(d_{*}(w,1),\ell_p)} \ \ \  \mbox{and}  \ \ \ B := \overline{\NA_G(d_{*}(w,1),\ell_p)}.
		\end{equation*} 
		As the set $A$ is $\alpha$-lineable by Theorem \ref{thmB} and $A\cap B = \emptyset$, it remains to prove that $A$ is stronger than $B$.

		Arguing as in the proof of Corollary \hyperref[corbigA2]{A2.(I)}, the claim can be proved. For the sake of completeness, we include the proof. If $R$ and $S$ are elements of $\NA_G(d_{*}(w,1),\ell_p)$, then there exists $N \in \mathbb{N}$ such that $(R+S)(e_n) = 0$ for every $n > N$. Hence, $R+S$ depends only on finitely many coordinates; thus $R+S \in \NA_G(d_{*}(w,1),\ell_p)$. It follows that $B+B \subseteq B$. This proves that $A$ is stronger than $B$. \end{proof}
        

	\subsection{Operators with values in $\ell_p (\Gamma)$}
	
In this section, we focus on result \ref{M3} from the Main Results and present new examples along the same lines. As mentioned before, Pellegrino and Teixeira \cite{PT} proved that whenever $X$ and $Y$ are Banach spaces such that $Y$ contains an isometric copy of $\ell_p$ for some $1 \leq p < \infty$, and $x_0 \in S_X$ is fixed, the non-linear subset
	\begin{equation*}\label{eq:NAx}
		\NA^{x_0} (X,Y) := \{ T  \in \NA(X, Y): \|T(x_0)\| = \|T\|\}
	\end{equation*}
is lineable in $\mathcal{L} (X,Y)$ (see \cite[Proposition 6]{PT}). Moreover, they showed that if $\NA (X,Y) \neq \mathcal{L}(X,Y)$, then the set $\mathcal{L}(X,Y) \setminus \NA (X,Y)$ is lineable (see \cite[Proposition 7]{PT}) under the same assumptions.

In this section, our aim is to generalize the aforementioned results in two directions. First, we consider these results in the setting of group-invariant operators. Second, the index set $\mathbb{N}$ of the space $\ell_p$ is replaced by an arbitrary infinite set $\Gamma$.

Let us also mention that we have already dealt with operators whose range lies in $\ell_p$. Namely, Theorem \ref{thmB} concerns operators from $d_*(w,1)$ into $\ell_p$ that cannot \emph{even} be approximated by norm-attaining operators. That is,
\[
\Lin_G (d_* (w,1),\ell_p) \setminus \overline{\NA_G (d_* (w,1),\ell_p)} 
\]
is $\mathfrak c$-spaceable. Although this result has the advantage of dealing with operators outside the closure of $\NA_G (d_* (w,1),\ell_p)$, it applies to the specific domain space $d_* (w,1)$. By contrast, the main result of this section applies to an arbitrary Banach space $X$, and concerns the complement of $\NA_G (X, \ell_p (\Gamma))$.

Before presenting and proving our results, we note that, for a {\it compact} group $G \subseteq \mathcal{L}(X)$ and any Banach space $Y$, the set
\begin{equation*} 
		\NA_G^{x_0}(X,Y) := \{ T \in \NA_G (X,Y) : \|T(x_0) \| = \|T\| \}
	\end{equation*} 
is nonempty by the $G$-invariant Hahn–Banach theorem (see \cite[Proposition 1]{F}). 

 	
	\begin{theorem}\label{thm:lqgamma}
		Let $X$ be a Banach space, $G \subseteq \mathcal{L}(X)$ a group of isometries, $\Gamma$ an infinite set, and $1 \leq p < \infty$.  
		\begin{itemize}
  \itemsep0.3em
			\item[\textup{(i)}] If $G$ is compact and $x_0 \in X_G$ with $\|x_0\|=1$, then $\ell_p (\Gamma)$ is isometrically contained in $\NA_G^{x_0} (X, \ell_p (\Gamma))$. 
			\item[\textup{(ii)}] If $\NA_G (X, \ell_p (\Gamma)) \neq \Lin_G(X,\ell_p(\Gamma))$, then $\ell_p (\Gamma)$ is isometrically contained in $(\Lin_G(X,\ell_p(\Gamma)) \setminus \NA_G (X, \ell_p (\Gamma)))\cup \{0\}$.
		\end{itemize}  
	\end{theorem} 		
	
	\begin{proof}    (i): Let $T \in \NA_G^{x_0}(X, \ell_p(\Gamma))$ be a nonzero operator. 
        \begin{itemize}
            \item Suppose that $|\Gamma| > \aleph_0$ and take $\Gamma_0 \subseteq \Gamma$ nonempty with $|\Gamma_0| \leq \aleph_0$ such that 
\begin{equation}\label{normx0}
    \|T\| = \|T(x_0)\|_p = \left( \sum_{\lambda \in \Gamma_0} |(T(x_0))(\lambda)|^p \right)^{1/p}.
\end{equation}
If $|\Gamma_0| < \aleph_0$, then $\Gamma_0 = \{ \lambda_1, \ldots, \lambda_l\}$ for some $l \in \N$.
In this case, we may choose $(\lambda_j)_{j={l+1}}^\infty$ arbitrarily in $\Gamma \setminus \Gamma_0$, since \eqref{normx0} implies that $T(x_0)(\lambda_j)=0$ for every $j > l$. Hence, by extending $\Gamma_0$ to $\Gamma_0 \cup \{\lambda_j: j \geq l+1\}$ if necessary, we may assume that $\Gamma_0 \subseteq \Gamma$ satisfies $|\Gamma_0|=\aleph_0$ and that \eqref{normx0} still holds. 
Let us write $\Gamma_0 = \{ \lambda_j: j \in \N \}$. 
            \item If $|\Gamma|=\aleph_0$, then write $\Gamma = \{ \lambda_j : j \in \mathbb{N}\}$ and set $\Gamma_0 := \emptyset$.
        \end{itemize}

Thus, in both cases we have that
\begin{equation}\label{eq:gamma_0}
    \|T(x_0)\|_p = \left( \sum_{j=1}^{\infty} |(T(x_0))(\lambda_j)|^p \right)^{1/p}.
\end{equation}

Write $\Gamma \setminus \Gamma_0 = \bigcup_{i \in \Lambda} \Gamma_i$ in such a way that $|\Gamma_i| = \aleph_0$ for every $i \in \Lambda$ and $\Gamma_i \cap \Gamma_j = \emptyset$ for any pair of distinct $i,j\in \Lambda$. By standard cardinal arithmetic for infinite cardinals (under AC), we have $|\Gamma \setminus \Gamma_0| = |\Gamma| = |\Lambda|$. For $i \in \Lambda$, we write 
\begin{itemize} 
\item $\Gamma_i := \{ \lambda_j^{(i)} : j \in \mathbb{N}\}$, 
\item $\ell_p^{(i)} := \{ x \in \ell_p (\Gamma) : x(\lambda) = 0 \text{ if } \lambda \not\in \Gamma_i \}. $
\end{itemize} 
Define $T_i \in \Lin (X, \ell_p^{(i)})$ by 
\begin{equation*} 
(T_i (x)) (\lambda_j^{(i)} ) = (T(x))(\lambda_j) \quad (j\in\mathbb{N}).
\end{equation*}
Then $T_i$ is $G$-invariant. Let us denote by $I^{(i)}$ the canonical embedding of $\ell_p^{(i)}$ into $\ell_p (\Gamma)$ and consider 
\begin{equation*} 
V_i := I^{(i)} \circ T_i \in \Lin (X, \ell_p (\Gamma)),
\end{equation*}
which is still $G$-invariant. Notice that $V_i(x_0)(\lambda_j^{(i)}) = I^{(i)} \circ T_i (x_0) (\lambda_j^{(i)}) = T(x_0)(\lambda_j)$. Let $(a_i)_{i \in \Lambda} \in \ell_p(\Gamma)$ be given. The linear operator $\sum_{i \in \Lambda} a_i V_i$ defined from $X$ into $\ell_p(\Gamma)$ is well-defined. In fact, given $x \in X$ and $\lambda \in \Gamma$, we have that  
\begin{equation*} 
\left( \sum_{i \in \Lambda} a_i V_i \right) (x) (\lambda) =  \begin{cases}
a_i V_i(x)(\lambda), & \text{if} \ \lambda \in \Gamma_i \ \mbox{for some} \ i \in \Lambda, \\
0, & \text{otherwise}.
\end{cases}
\end{equation*} 
Then $\sum_{i \in \Lambda} a_i V_i$ is $G$-invariant.

We claim that 
$\sum_{i \in \Lambda} a_i V_i$ attains its norm at $x_0$ and 
\[
\left\|  \sum_{i \in \Lambda} a_i V_i \right\| = \|(a_i)_{i \in \Lambda}\|_{\ell_p(\Lambda)} \|T\|.
\]
Indeed, suppose that 
\begin{itemize}
    \item $F_1 \subseteq \Gamma_{j_1}, \ldots, F_m \subseteq \Gamma_{j_m}$ are finite subsets for some $j_1,\ldots, j_m$ of $\Lambda$.
    \item (Enlarging each set, if necessary) write $F_k = \{ \lambda_1^{(j_k)}, \ldots, \lambda_{l}^{(j_k)} \}$ for each $k=1,\ldots,m$, where $l$ is a positive integer. 
\end{itemize}
If we put $F = F_1 \cup \cdots \cup F_m$, then 
\begin{align*}
      \sum_{\lambda \in F}  \left| \sum_{i \in \Lambda} a_i V_i(x)(\lambda) \right|^p &= \sum_{k=1}^m \sum_{\lambda \in F_k} \left| \sum_{i \in \Lambda} a_i V_i(x)(\lambda) \right|^p \\ 
      &=\sum_{k=1}^m \sum_{i=1}^{l} |a_{j_k}|^p |T(x) (\lambda_i)|^p = \sum_{k=1}^m |a_{j_k}|^p \sum_{i=1}^{l} |T(x)(\lambda_i)|^p.
\end{align*}
This shows that $\| \sum_{i \in \Lambda} a_i V_i (x)\|_{\ell_p (\Gamma)} \leq  \|(a_i)_{i \in \Lambda}\|_{\ell_p(\Lambda)} \|T(x)\|$. On the other hand, notice that $j_1,\ldots, j_m$ are chosen arbitrarily from $\Lambda$ and $l$ is an arbitrary positive integer. It follows that 
\begin{align*} 
			\left\| \sum_{i \in \Lambda} a_i V_i (x_0) \right\|_{\ell_p(\Gamma)}^p  &=  \sup \left\{ \sum_{\gamma \in F} \left| \sum_{i \in \Lambda} a_i V_i(x_0)(\gamma) \right|^p: F \subseteq \Gamma \ \mbox{finite set} \right\}  \\ 
   & \geq \sup \left\{ \sum_{k=1}^m |a_{j_k}|^p \sum_{i=1}^{l} |T(x_0)(\lambda_i)|^p: j_1, \ldots, j_m \in \Lambda, \, l \in \mathbb{N} \right\} \\
   & = \|(a_i)_{i \in \Lambda}\|_{\ell_p(\Lambda)}^p \|T(x_0)\|^p
		\end{align*} 
where the last equality holds because of \eqref{eq:gamma_0}. 
This proves the claim, and shows that $\{V_i : i \in \Lambda\}$ generates $\ell_p (\Lambda)$ (isometrically, $\ell_p (\Gamma)$) inside $\NA_G^{x_0} (X, \ell_p (\Gamma))$.

(ii): Let $T \in \mathcal{L}(X, \ell_p(\Gamma)) \setminus \NA(X, \ell_p(\Gamma))$ be given. 
\begin{itemize}
    \item Suppose that $|\Gamma | > \aleph_0$. Take a sequence $(x_n) \subseteq B_X$ such that $\|T(x_n)\| \rightarrow \|T\|$ as $n \rightarrow \infty$. For each $n \in \N$, choose a subset $\Gamma_n \subseteq \Gamma$ with $|\Gamma_n| \leq \aleph_0$ such that 
\begin{equation*} 
    \|T(x_n)\| = \left( \sum_{\lambda \in \Gamma_n} |(Tx_n)(\lambda)|^p \right)^{1/p}.
\end{equation*}
Let $\Gamma_0 = \bigcup_{n \in \N} \Gamma_n$ and write $\Gamma_0 = \{\lambda_j : j \in \mathbb{N}\}$.
    \item If $|\Gamma|=\aleph_0$, then write $\Gamma = \{ \lambda_j : j \in \mathbb{N}\}$ and set $\Gamma_0:=\emptyset.$
\end{itemize}
In both cases, we obtain that 
\begin{equation*}
    \sup_{x \in B_X} \left( \sum_{j=1}^{\infty} |(Tx)(\lambda_j)|^p \right)^{1/p} = \|T\|. 
\end{equation*}
As in part (i), consider 
\begin{itemize}
    \item $\Gamma \setminus \Gamma_0 = \bigcup_{i \in \Lambda} \Gamma_i$ such that $|\Gamma_i| = \aleph_0$ for every $i \in \Lambda$ and $\Gamma_i \cap \Gamma_j = \emptyset$ for any pair of distinct $i,j\in \Lambda$. Moreover, $|\Gamma \setminus \Gamma_0| = |\Gamma| = |\Lambda|$.
    \item $\Gamma_i = \{ \lambda_j^{(i)}: j \in \N\}$ for each $i \in \Lambda$.
    \item $\{V_i: i \in \Lambda\} \subseteq \mathcal{L}(X, \ell_p (\Gamma))$.
\end{itemize}
  Given $(a_i)_{i \in \Lambda} \in \ell_p (\Lambda)$, the computations in part (i) yield that
\begin{equation*}
    \left\| \sum_{i \in \Lambda} a_i V_i \right\|^p = \|(a_i)_{i \in \Lambda} \|_{\ell_p(\Lambda)}^p \|T\|^p > \|(a_i)_{i \in \Lambda}\|_{\ell_p(\Lambda)}^p \|T(x)\|^p \geq \left\| \sum_{i \in \Lambda} a_i V_i(x) \right\|^p
\end{equation*}
for every $x \in B_X$, since $T$ does not attain its norm. 
Thus, $\ell_p(\Lambda)$ is isometrically embedded in $(\mathcal{L}_G(X, \ell_p(\Gamma)) \setminus \NA_G(X, \ell_p(\Gamma)))\cup\{0\}$.
	\end{proof}

	\subsection{Operators with values in $c_0 (\Gamma)$}

	In the following result, we provide a version of Theorem \ref{thm:lqgamma} for $c_0(\Gamma)$-spaces. To the best of our knowledge, this is not known even in the classical case (that is, when $G= \{ \text{Id}\}$). 
	
	\begin{theorem}\label{thm:c_0-target}
		Let $X$ be a Banach space, $G \subseteq \mathcal{L}(X)$ a group of isometries, and $\Gamma$ an infinite set.  
		\begin{itemize}
  \itemsep0.3em
			\item[\textup{(i)}] If $G$ is compact and $x_0 \in X_G$ with $\|x_0\|=1$, then $c_0 (\Gamma) $ is isometrically contained in $\NA_G^{x_0} (X, c_0 (\Gamma))$.
			\item[\textup{(ii)}] If $\NA_G (X, c_0 (\Gamma)) \neq \Lin_G(X,c_0 (\Gamma))$, then $c_0 (\Gamma)$ is isometrically contained in $(\Lin_G(X,c_0(\Gamma)) \setminus \NA_G (X, c_0 (\Gamma)))\cup\{0\}$. 
		\end{itemize}  
		
	\end{theorem} 
	
	\begin{proof}
	(i) Fix a non-zero operator $T \in \NA_G^{x_0} (X, c_0(\Gamma))$. Then the set given by 
		\begin{equation*} 
			\Gamma_0:=\{\lambda\in\Gamma: |(Tx_0)(\lambda)| > \|T\|/2\} 
		\end{equation*} 
		is finite, say $\Gamma_0 =\{\lambda_1,\ldots,\lambda_\ell\}$ for some $\ell \in \mathbb{N}$. We can write 
		\begin{equation*} 
			\Gamma\setminus\Gamma_0=\bigcup_{i\in\Lambda} \Gamma_i
		\end{equation*} 
		as a union of pairwise disjoint sets $\Gamma_i$ with $|\Gamma_i|=\aleph_0$ and $|\Lambda| = |\Gamma|$. 
		For each $i \in \Lambda$, we consider
		\[
		c_0^{(i)} := \{ x \in c_0 (\Gamma) : x(\lambda) = 0 \text{ if } \lambda \not\in \Gamma_i \}. 
		\]
		Write $\Gamma_i=\{ \lambda_j^{(i)} : j \in \mathbb{N}\}$. Define also $T_i \in \Lin (X, c_0^{(i)} )$ by 
		\begin{equation}\label{eq:Tic0}
			(T_i (x)) (\lambda_j^{(i)} ) = (T(x))(\lambda_j) \quad ( 1 \leq j \leq \ell )
		\end{equation}
		and $(T_i (x)) (\lambda_j^{(i)} ) = 0$ for $j > \ell$, and $T_i (x) (\lambda) =0$ if $\lambda \not\in \Gamma_i$. 
		
        For each $i \in \Lambda$, consider $V_i := I^{(i)} \circ T_i \in \Lin (X, c_0 (\Gamma))$, where $I^{(i)}$ is the canonical embedding of $c_0^{(i)}$ into $c_0 (\Gamma)$. Then $V_i$ is $G$-invariant, and $\|V_i\| = \|V_i (x_0)\| = \|T(x_0) \| = \|T\|$ for each $i\in\Lambda$ by \eqref{eq:Tic0}.

Let $(a_i)_{i\in \Lambda} \in c_0(\Lambda)$ and define $\sum_{i \in \Lambda} a_i V_i \in \mathcal{L}(X, c_0 (\Gamma))$ by 
\begin{equation}\label{eq:c0_sum_V_i's}
\left( \sum_{i \in \Lambda} a_i V_i \right) (x)(\lambda) = \begin{cases}
a_i V_i(x)(\lambda), & \text{if} \ \lambda \in \Gamma_i \ \mbox{for some} \ i \in \Lambda, \\
0, & \text{otherwise}.
\end{cases}
\end{equation}
Note that $\sum_{i \in \Lambda} a_i V_i$ is $G$-invariant. We claim that $\sum_{i \in \Lambda} a_i V_i$ attains its norm at $x_0$.
Indeed, 
\begin{equation*}
    \left|\left(\sum_{i \in \Lambda} a_i V_i \right) (x)(\lambda_j^{(k)} )\right| = |a_k (Tx)(\lambda_j)| = |a_k| |(Tx)(\lambda_j)|
\end{equation*}
for every $x\in B_X$, $j \in \mathbb{N}$ and $k \in \Lambda$;
hence 
\begin{equation*} 
\left \|\left(\sum_{i \in \Lambda} a_i V_i \right)  (x_0) \right \| = \|a\|_{c_0 (\Lambda)}  \|T \|.
\end{equation*} 
        It follows that $\{V_i : i \in \Lambda\}$ generates an isometric copy of $c_0 (\Lambda)$ in $\NA_G^{x_0} (X, c_0 (\Gamma))$.
		
		(ii): 
    Pick $T \in \Lin_G(X,c_0(\Gamma)) \setminus \NA_G (X, c_0 (\Gamma))$. As in the proof for part (ii) of Theorem \ref{thm:lqgamma}, 
    \begin{itemize}
       \item If $|\Gamma | > \aleph_0$, then take a sequence $(x_n) \subseteq B_X$ such that $\|T(x_n)\| \rightarrow \|T\|$ as $n \rightarrow \infty$. For each $n \in \N$, choose $\lambda_n \in \Gamma$ such that $\|T(x_n)\| = |T(x_n)(\lambda_n)|$. Write $\Gamma_0 =\{ \lambda_j : j\in \mathbb{N}\}$.
    \item If $|\Gamma|=\aleph_0$, then write $\Gamma = \{ \lambda_j : j \in \mathbb{N}\}$ and set $\Gamma_0:=\emptyset.$
        \item In either case, that is, whether $|\Gamma|>\aleph_0$ or $|\Gamma|=\aleph_0$, write $\Gamma \setminus \Gamma_0 = \bigcup_{i \in \Lambda} \Gamma_i$ such that $|\Gamma_i| = \aleph_0$ for every $i \in \Lambda$ and $\Gamma_i \cap \Gamma_j = \emptyset$ for any pair of distinct $i,j\in \Lambda$. Moreover, $|\Gamma \setminus \Gamma_0| = |\Gamma| = |\Lambda|$.
    \item Write $\Gamma_i = \{ \lambda_j^{(i)}: j \in \N\}$ for each $i \in \Lambda$.
\end{itemize}
Consider $c_0^{(i)} := \{ x \in c_0 (\Gamma) : x(\lambda) = 0 \text{ if } \lambda \not\in \Gamma_i \}$ and define $T_i \in \Lin (X, c_0^{(i)} )$ for each $i \in \Lambda$ by 
		\begin{equation}\label{eq:Tic02}
			(T_i (x)) (\lambda_j^{(i)} ) = (T(x))(\lambda_j) \quad (j \in \mathbb{N}).
		\end{equation}
		Put $V_i := I^{(i)} \circ T_i \in \Lin (X, c_0 (\Gamma))$, where $I^{(i)}$ is the canonical embedding of $c_0^{(i)}$ into $c_0 (\Gamma)$. Note that $\|V_i\| = \|T_i\|$ and $V_i$ is $G$-invariant. Moreover, by \eqref{eq:Tic02} and recalling that $T$ is non-norm-attaining, we get that 
		\[
		\|V_i (x)\| \leq \|T(x)\| < \|T\|=\|V_i\|. 
		\]
		Thus, $V_i$ does not attain its norm.

Given $(a_i)_{i\in \Lambda} \in c_0(\Lambda)$, consider $\sum_{i \in \Lambda} a_i V_i \in \mathcal{L}_G (X, c_0 (\Gamma))$ as in \eqref{eq:c0_sum_V_i's}. We claim that $\sum_{i \in \Lambda} a_i V_i \in \mathcal{L}_G (X, c_0 (\Gamma))$ does not attain its norm. Assume to the contrary that it attains its norm at some $x_0 \in B_X$. Then there exists $\lambda \in \Gamma$ such that 
\begin{equation*}
    \left| \left( \sum_{i \in \Lambda} a_i V_i \right) (x_0) (\lambda) \right| = \|a\|_{c_0(\Lambda)} \|T\|.
\end{equation*}
This implies that there exist $i \in \Lambda$ and $j \in \mathbb{N}$ such that $\lambda \in \Gamma_i$ with $\lambda = \lambda_j^{(i)}$ and $|a_{i}(Tx_0)(\lambda_j)| = \|a\|_{\infty} \|T\|$. It follows that $T$ attains its norm at $x_0$, which is a contradiction. Thus, $\{V_i : i \in \Lambda\}$ generates an isometric copy of $c_0 (\Lambda)$ in $(\mathcal{L}_G (X, c_0(\Gamma)) \setminus \NA_G (X, c_0(\Gamma)))\cup\{0\}$. 	
\end{proof} 

 We finish this section by presenting consequences of Theorems \ref{thm:lqgamma} and \ref{thm:c_0-target}.

 \begin{corollary}\label{cor:lqc0}
Let $X$ be a Banach space and $\Gamma$ an infinite set. Fix $x_0 \in S_X$ and $1 \leq p < \infty$. Then,
    \begin{itemize}
        \itemsep0.3em
        \item[\textup{(i)}] $\NA^{x_0}(X,\ell_p(\Gamma))$ is $ |\Gamma|^{\aleph_0}$-spaceable in $\mathcal{L}(X,\ell_p(\Gamma))$. 
        \item[\textup{(ii)}] $\NA^{x_0} (X, c_0 (\Gamma))$ is $ |\Gamma|^{\aleph_0}$-spaceable in $\Lin (X,c_0(\Gamma))$.
        \item[\textup{(iii)}] If $\Lin (X,\ell_p(\Gamma)) \setminus \NA (X, \ell_p (\Gamma))$ is nonempty, then $\Lin (X,\ell_p(\Gamma)) \setminus \NA (X, \ell_p (\Gamma))$ is $ |\Gamma|^{\aleph_0}$-spaceable in $\Lin (X,\ell_p(\Gamma))$.
        \item[\textup{(iv)}] If $\Lin (X,c_0(\Gamma)) \setminus \NA (X, c_0 (\Gamma))$ is nonempty, then $\Lin (X,c_0(\Gamma)) \setminus \NA (X, c_0 (\Gamma))$ is $ |\Gamma|^{\aleph_0}$-spaceable in $\Lin (X,c_0(\Gamma))$. 
    \end{itemize}
 \end{corollary}

 \noindent \textbf{Acknowledgements}. 
This work was partially carried out at the University of Valencia during a research stay of Daniel L. Rodríguez-Vidanes, who acknowledges the support provided by the Department of Mathematical Analysis of the Faculty of Mathematical Sciences at the University of Valencia. 
The authors would like to thank Helena Del R\'io for several fruitful conversations on the topic of this manuscript.

\vspace{0.2cm} 
 \noindent 
\textbf{Funding information}: Sheldon Dantas and Javier Falcó were both supported by grant PID2021-122126NB-C33 funded by MICIU/AEI/10.13039/501100011033 and by ERDF/EU. Sheldon Dantas was also supported by Grant PID2021-122126NB-C31 funded by MICIU/AEI/10.13039/501100011031 and by ERDF/EU. Mingu Jung was supported by a KIAS Individual Grant (MG086601), by June E Huh Center for Mathematical Challenges (HP086601) at Korea Institute for Advanced Study, and by the research fund of Hanyang University (HY-202500000003346). Daniel L. Rodríguez-Vidanes was supported by Grant PGC2018-097286-B-I00 and by the Spanish Ministry of Science, Innovation and Universities and the European Social Fund through a “Contrato Predoctoral para la Formación de Doctores, 2019” (PRE2019-089135).

\end{document}